\documentclass[11pt,twoside,final]{article}

\usepackage{hyperref}
\hypersetup{colorlinks,linkcolor={blue},citecolor={blue},urlcolor={red}} 

\usepackage{fullpage}

\usepackage[mathscr]{eucal}
\usepackage{mathrsfs}
\usepackage{amsbsy}

\DeclareMathAlphabet{\mathpzc}{OT1}{pzc}{m}{it}

\usepackage{epsf}
\usepackage{fancyheadings}
\usepackage{graphics}
\usepackage{graphicx}
\usepackage{enumitem}
\usepackage{float}

\usepackage{color}

\usepackage{amsthm}
\usepackage{amsfonts}
\usepackage{amsmath}
\usepackage{amssymb}

\usepackage{algorithm}
\usepackage{mathabx}
\usepackage{algorithmic}
\usepackage{bbm}

\usepackage{caption}
\usepackage{subcaption}

\usepackage[giveninits=true,  
  maxnames = 2,  
  backref=true,
  backend=bibtex,
  style=alphabetic,
  sortlocale=de_DE,
  natbib=true,
  doi=false,
  isbn=false,
  url=false]{biblatex}

\newlist{cdesc}{description}{1}
\setlist[cdesc]{font=\mdseries,itemsep=0.5pt}

 \usepackage{macros}
\usepackage[noabbrev,nameinlink]{cleveref}

\newcommand{\figdir}{figures/}

\newcommand{\Covariate}{\ensuremath{X}}

\newcommand{\Response}{\ensuremath{Y}}

\newcommand{\Xspace}{\ensuremath{\mathcal{X}}}
\newcommand{\Wspace}{\ensuremath{\mathcal{W}}}

\newcommand{\ERIC}{\mbox{\texttt{PAST}}}

\renewcommand{\numobs}{\tnumobs}

\newcommand{\tarplain}{\ensuremath{f}}
\newcommand{\tarstar}{\ensuremath{{\tarplain^{*}}}}
\newcommand{\tarhat}{\ensuremath{\widehat{\tarplain}}}
\newcommand{\TarClass}{\ensuremath{\mathcal{F}}}

\newcommand{\auxplain}{\ensuremath{g}}
\newcommand{\auxstar}{\ensuremath{{\auxplain^{*}}}}
\newcommand{\auxhat}{\ensuremath{\widetilde{\auxplain}}}

\newcommand{\Surrogate}{\ensuremath{W}}

\newcommand{\surrogate}{\ensuremath{w}}

\newcommand{\Dataset}{\dataset}

\renewcommand{\Response}{\ensuremath{Y}}
\newcommand{\Yspace}{\ensuremath{\mathcal{Y}}}

\newcommand{\ResponseTil}{\ensuremath{\widetilde{\Response}}}
\newcommand{\DatasetTil}{\ensuremath{\widetilde{\Dataset}}}

\newcommand{\LabelSet}{\ensuremath{\Dataset_{\scaleto{\mathrm{L}}{5pt}}}}
\newcommand{\UnlabelSet}{\ensuremath{\Dataset_{\scaleto{\mathrm{U}}{5pt}}}}

\newcommand{\genfun}{\ensuremath{f}}
\newcommand{\CompFun}{\ensuremath{H}}

\newcommand{\myrad}{\ensuremath{t}}

\newcommand{\totempnorm}[1]{\ensuremath{\|#1\|}_{\numobs}}

\newcommand{\TarClassStar}{\ensuremath{\TarClass^*}}

\newcommand{\Pemp}{\ensuremath{\Prob}}
\newcommand{\PempTot}{\ensuremath{\Pemp_\numobs}}

\newcommand{\PempLab}{\ensuremath{\Pemp_\lnumobs}}
\newcommand{\PempUnlab}{\ensuremath{\Pemp_\unumobs}}
\newcommand{\EmpNormUnlab}[1]{\ensuremath{\|#1\|_\unumobs}}

\newcommand{\Logfun}{\ensuremath{\varphi}}

\newcommand{\uradcrit}{\ensuremath{r_\unumobs}}
\newcommand{\lradcrit}{\ensuremath{r_\lnumobs}}

\newcommand{\VarFun}[1]{\ensuremath{\sigma^2(#1)}}

\newcommand{\NewEmp}{\ensuremath{V}}

\newcommand{\HighOrder}{\ensuremath{\tau}}

\newcommand{\revdefn}{\ensuremath{= \, \colon}}

\newcommand{\yhat}{\ensuremath{\widehat{y}}}

\newcommand{\Ytil}{\ensuremath{\widetilde{Y}}}

\newcommand{\Proc}{\ensuremath{\mathcal{P}}}

\newcommand{\Loss}{\loss}
\newcommand{\Tarclass}{\TarClass}

\newcommand{\Fclass}{\TarClass}
\newcommand{\Gclass}{\mathcal{G}}

\newcommand{\Lip}{\ensuremath{L}}
\newcommand{\scparam}{\ensuremath{\gamma}}

\newcommand{\hstar}{\ensuremath{h^*}}

\newcommand{\Labelset}{\LabelSet}
\newcommand{\Unlabelset}{\UnlabelSet}

\newcommand{\yzero}{y_0}
\newcommand{\yone}{y_1}

\newcommand{\yhatzero}{\yhat_0}
\newcommand{\yhatone}{\yhat_1}

\newcommand{\tardagger}{\ensuremath{f^\dagger}}

\newcommand{\sigbound}{\ensuremath{\sigma}}

\newcommand{\FullDataset}{\ensuremath{\widetilde{\Dataset}}}
\newcommand{\LossDiff}{\ensuremath{\mathcal{D}}}

\newcommand{\gammahat}{\widehat{\gamma}}

\newcommand{\kdim}{\ensuremath{k}}
\newcommand{\usedim}{\ensuremath{d}}

\newcommand{\gnoise}{\ensuremath{\varepsilon}}
\newcommand{\Normal}{\ensuremath{\mathcal{N}}}
\newcommand{\usedima}{\ensuremath{{\usedim_1}}}
\newcommand{\usedimb}{\ensuremath{{\usedim_2}}}

\newcommand{\CovMat}{\ensuremath{\mathbf{\Sigma}}}

\newcommand{\alphastar}{\ensuremath{\alpha^*}}

\newcommand{\Feat}{\ensuremath{\Psi}}

\newcommand{\GLMbias}{\|\widetilde{h} - \hstar \|_2}

\newcommand{\htil}{\widetilde{h}}
\newcommand{\auxfun}{g}

\usepackage{scalerel}

\newcommand{\Wprob}{\smallsub{h}{W}}
\newcommand{\Zprob}{\smallsub{h}{Z}}
\newcommand{\Yprob}{\smallsub{h}{Y}}

\newcommand{\HardY}{\smallsub{\ResponseTil}{\mbox{hard}}}
\newcommand{\SoftY}{\smallsub{\ResponseTil}{\mbox{soft}}}

\newcommand{\sigmoid}{\ensuremath{\varphi}}

\newcommand{\fsoft}{\ensuremath{\smallsub{\ftil}{\mbox{soft}}}}
\newcommand{\fhard}{\ensuremath{\smallsub{\ftil}{\mbox{hard}}}}

\newcommand{\tnow}{\ensuremath{\tinysub{t}{\mbox{now}}}}
\newcommand{\tfuture}{\ensuremath{\smallsub{t}{\mbox{fut}}}}
\newcommand{\tinter}{\ensuremath{\smallsub{t}{\mbox{int}}}}
\newcommand{\Zcovv}[2]{\ensuremath{Z_{#2}^{#1}}}
\newcommand{\Zcov}[1]{\ensuremath{Z^{#1}}}

\newcommand{\Yhat}{\ensuremath{\widehat{Y}}}

\begin{document}


\begin{center}

  {\bf{\LARGE{Prediction Aided by Surrogate Training}}}
\vspace*{.2in}

{\large{
\begin{tabular}{ccc}
Eric Xia$^{\dagger}$ & Martin
J. Wainwright$^{\dagger, \ddagger}$
\end{tabular}
}}

\vspace*{.2in}

\begin{tabular}{c}
  Laboratory for Information and Decision Systems \\ Statistics and
  Data Science Center \\ Electrical Engineering \& Computer
  Sciences$^\dagger$, and Mathematics$^\ddagger$ \\
  Massachusetts Institute of Technology, Cambridge, MA
\end{tabular}

\medskip

\today

\vspace*{.2in}

\begin{abstract}
We study a class of prediction problems in which relatively few
observations have associated responses, but all observations include
both standard covariates as well as additional ``helper'' covariates.
While the end goal is to make high-quality predictions using only the
standard covariates, helper covariates can be exploited during
training to improve prediction.  Helper covariates arise in many applications, including
forecasting in time series; incorporation of biased or mis-calibrated
predictions from foundation models; and sharing information in
transfer learning.  We propose ``prediction aided by surrogate
training'' (\texttt{PAST}), a class of methods that exploit labeled
data to construct a response estimator based on both the standard and
helper covariates; and then use the full dataset with pseudo-responses
to train a predictor based only on standard covariates.  We establish
guarantees on the prediction error of this procedure, with the
response estimator allowed to be constructed in an arbitrary way, and
the final predictor fit by empirical risk minimization over an
arbitrary function class.  These upper bounds involve the risk
associated with the oracle data set (all responses available), plus an
overhead that measures the accuracy of the pseudo-responses.  This
theory characterizes both regimes in which \texttt{PAST} accuracy is
comparable to the oracle accuracy, as well as more challenging regimes
where it behaves poorly.  We demonstrate its empirical performance
across a range of applications, including forecasting of societal ills
over time with future covariates as helpers; prediction of
cardiovascular risk after heart attacks with prescription data as
helpers; and diagnosing pneumonia from chest X-rays using
machine-generated predictions as helpers.
\end{abstract}

\end{center}


\section{Introduction}

In modern data science, it is often expensive and/or time-consuming to
collect labels or responses for solving a classification or regression
problem.  For instance, in medical settings, collecting labeled data
requires substantial time investment from doctors and other
experts~\cite{rajkomar2019machine}.  Similarly, in longitudinal
studies, participants may stop responding to follow-up
surveys~\cite{hogan2004handling}, which again leads to missing
responses.  In these settings, it is natural to consider datasets of a
hybrid type, partitioned into one collection of labeled responses, and
another unlabeled set.  Typically, the labeled set is substantially
smaller than the unlabeled set, for the reasons described above.  In
the machine learning (ML) literature, problems involving datasets of
this hybrid type are known as instances of semi-supervised learning
(e.g.,~\cite{chapelle2009semi}); in statistics, related problems have
have been formulated and tackled using the EM algorithm
(e.g.,~\cite{mclachlan2008algorithm, wu1983convergence}), as well as
various semiparametric techniques (e.g.,~\cite{robins1994estimation,
  robins1995semiparametric}). See~\Cref{SecRelated} for further
discussion of this and other related work.

The focus of this paper is a variant of the standard semi-supervised
set-up, distinguished by the availability of additional ``helper''
covariates.  In more explicit terms, the standard set-up is characterized by
a covariate $\Covariate \in \Xspace$, a scalar response $\Response
\in \Yspace$, with a the labeled dataset consisting of covariate-response
pairs $(\Covariate, \Response)$, and an unlabeled dataset
consisting of covariates $\Covariate$ alone.  In this paper, we augment
this set-up with the availability of an additional random vector
$\Surrogate \in \Wspace$---known as the \emph{helper} or \textit{surrogate covariate}---that
is collected for all samples.  Thus, the labeled dataset now consists
of triples $(\Covariate, \Surrogate, \Response)$ whereas the unlabeled
data set consists of the pairs $(\Covariate, \Surrogate)$.  As we
discuss in~\Cref{SecHelper}, there are a wide range of problems that
can be fruitfully cast in these terms.

Within this set-up, our goal is to estimate---typically over some
non-parametric class---a function $\covariate \mapsto
\fhat(\covariate)$ that is a good predictor of the response.  We
formalize the notion of ``goodness'' more precisely in terms of risk
minimization (cf. equation~\eqref{EqnPopulationRisk}).  In many cases,
the quality of $\fhat$ can be measured via its proximity to the
regression function \mbox{$\fstar(\covariate) \defn \Exs[\Response
    \mid \Covariate = \covariate]$.}  While our ultimate goal remains
making accurate predictions based \emph{only} on the original
covariate vector, the helper covariates can be exploited as part of
the training process.  In this paper, we propose and analyze a simple
third-stage procedure, one that integrates seamlessly with existing
pipelines for large-scale machine learning.  In the first stage, we
use the labeled triples $(\Covariate, \Surrogate, \Response)$ to
estimate a function $(\covariate, \surrogate) \mapsto
\auxhat(\covariate, \surrogate) \in \real$, then the second step uses that to generate
pseudo-responses $\Ytil$.  We then use the full dataset, augmented
with these pseudo-responses, to determine our final estimate $\fhat$,
which depends only on $\covariate$.

Our overall approach, which we refer to as ``prediction aided by
surrogate training'', or \ERIC~for short, should be understood as a
\emph{meta-procedure} since we allow a great deal of flexibility in
how the auxiliary function $\auxhat$ and final estimate $\fhat$ are
produced.  On the theoretical front, we provide explicit and
non-asymptotic bounds on the accuracy of $\fhat$ relative to the
optimal $\fstar$; see Theorems~\ref{thm:sqloss} and~\ref{thm:main}.
These bounds make a number of qualitative predictions about the
factors that control the accuracy of the final output $\fhat$;
see~\Cref{SecEnsemble,SecBinary} for such insights in the contexts of
least-squares regression and binary classification, respectively.  On
the empirical side, we demonstrate the effectiveness of the
\ERIC~procedure for a range of real-world prediction problems,
including forecasting societal ills and pneumonia detection from chest
X-rays, among others; see~\Cref{sec:empirical} for these results.

\subsection{How do helper covariates arise?}
\label{SecHelper}

There are many problems in which helper covariates arise in a natural
way.  Here we describe a few broad classes that motivated our work.

\paragraph{Noisy or mis-calibrated surrogate responses:}

As noted previously, in many applications, it is difficult and/or
costly to collect true responses $\Response$.  At the same time, it is
often the case that one has access to a procedure---for instance, a
pre-trained ML model, or some other predictive system---that can be
used to generate surrogate responses $\Response'$.  For example,
consider the problem of predicting diseases in medicine: it is
time-consuming and costly to have medical professionals produce
ground-truth labels, so researchers have resorted to automated
means of producing pseudo-labels.  One approach is to train natural
language models models or LLMs on electronic health
records~\cite{irvin2019chexpert}); other researchers have used
heuristics from domain experts to generate noisy labels for
identifying aortic valve malformations~\cite{fries2019weakly}.
Similarly, in protein folding, the Alpha-Fold
system~\cite{jumper2021highly} can be used to produce a large number
of pseudo-labels, which do not perfectly correspond with labels
determined by physical measurements~\cite{scardino2023good}.
Responses can also be collected by automated systems, such as in
crowd-sourcing~\cite{whitehill2009whose}; this approach underlied the
generation of the ImageNet dataset~\cite{deng2009imagenet}, which
catalyzed research in computer vision and deep learning.  Such
crowd-sourced responses, while often informative of the true response,
can be of variable quality and suffer from calibration issues.

Given a collection of surrogate responses, one naive approach is to
treat them as instances of the true response, and then to feed this
augmented data into the usual pipeline.  However, doing so directly
passes along any noise or mis-calibration in the surrogate responses,
and these discrepancies can lead the naive approach to behave very
poorly.\footnote{For example, in binary classification (where
$\Response \in \{0,1\}$), there might be non-trivial amounts of label
flipping in moving from $\Response$ to $\Response'$.  For real-valued
prediction problems, it could be that the surrogate response
$\Response'$ is close to $\Response$ after some type of
transformation.  Removing bias from $\Response'$ could be achieved by
a linear transformation.}  Within our framework, it is natural to
treat the noisy response as a helper covariate---i.e., setting
$\Surrogate = \Response'$.  The first stage of our procedure then
involves fitting a predictive model $\auxhat$ from pairs $(\Covariate,
\Response')$ to the true responses $\Response$.  As we discuss in this
sequel, this step allows our procedure to automatically adjust for
bias or mis-calibration in the generation of the pseudo-responses
$\Response'$, and moreover to exploit any additional information
available in the standard covariate $\Covariate$ in doing so.

\paragraph{Forecasting and longitudinal data:}  There are various
forms of data that have a temporal structure, among them longitudinal
data collected through surveys (e.g.,~\cite{diggle2002analysis,
  fitzmaurice2012applied}) and time-series models of dynamic
phenomena, including weather, stock prices, disease and rates
\mbox{(e.g.,~\cite{box2015time, brillinger2001time}).} More formally,
suppose that $\{\Zcov{t} \}_{t \geq 0}$ is a vector-valued time series
representing the phenomenon of interest, and our goal is to predict
the response $\Response \defn \Zcovv{\tfuture}{1}$ of the first
co-ordinate at some future time $\tfuture$, based on the features
$\Zcov{\tnow}$ available at the present time $\tnow$.  For such
problems, a natural choice of helper covariate $\Surrogate$ is the
feature vector $\Zcov{\tinter}$ measured at some intermediate time
$\tinter \in (\tnow, \tfuture)$.  Helper covariates of this type are
available in many applications.  For instance, in policy analysis, the
response $\Response$ might correspond to the effect of a job training
program at a time $9$ years after the conclusion of the program (so
that $\tfuture \defn \tnow + 9$).  While there might be relatively few
such measurements $9$ years in the future, it could be much easier to
obtain data on employment rates and other features in the 2 years
following the program (e.g., ~\cite{athey2019surrogate}).
See~\Cref{sec:BRFSS} for in-depth exploration of another example,
involving forecasting of societal ills such as alcoholism,

If we make the identifications $\Surrogate = \Zcov{\tinter}$ and
$\Covariate = \Zcov{\tnow}$, the auxiliary prediction problem that
determines $\auxhat$ becomes one of predicting the future response
$\Response$ at time $\tfuture$ based on feature pairs $(\Zcov{\tnow},
\Zcov{\tinter})$ from the current time $\tnow$ and an intermediate
time $\tinter$.  Since $\Zcov{\tinter}$ is closer in time to the
response $\Response$, we expect that this auxiliary prediction problem
should be substantially easier than that of predicting $\Response$
based on $\Zcov{\tnow}$ alone.

\paragraph{Transfer learning and distribution shift:}
Transfer learning refers to the problem of sharing information from
multiple (presumably related) tasks so as to improve predictive
performance (e.g.,~\cite{zhuang2020comprehensive, pan2009survey}).
The framework of helper covariates provides a natural
mechanism for sharing information from related tasks.  For example, in
personalized recommendation systems for online marketplaces, platforms
may have limited data on whether a given customer has purchased a
product (corresponding to the targeted response $\Response$).
However, they may have data and trained systems to predict whether or
not customers have clicked on the web-page for that
product~\cite{bastani2021predicting}; this information can be encoded
as a helper covariate $\Surrogate$.  In the closely related area of
distribution shift (e.g.~\cite{shimodaira2000improving, sugiyama2007covariate, ma2023optimally, pmlr-v75-kpotufe18a,
schmidthieber24transfer, reeve2021adaptive}), 
the goal is to build predictive models that are robust to changes in
the data-generating mechanisms; see the paper~\cite{Wilds21} for
survey of how it arises.  As one example, consider the problem of
predicting economic well-being based on satellite
imagery~\cite{Yeh20}; in this application, we might be interested in
predictions for one country (e.g., Angola) and use the responses of a
model trained on related country (e.g., Kenya) as a helper covariate.


\subsection{Related work}
\label{SecRelated}

Datasets that involve a combination of labeled and unlabeled samples
have been studied in different communities, using a variety of
techniques.  In machine learning, this set-up is referred to as
semi-supervised learning, and there is a rich associated literature
(e.g., see the survey~\cite{chapelle2009semi} and references therein).
Much of this work has focused on classification problems, contributing
both methodology (e.g., ~\cite{kulis2005semi, laine2017temporal}) as
well as theoretical insights (e.g.,~\cite{delalleau2005efficient,
  joachims1999transductive}).  Statistics also has a rich literature
on the related area of missing data, including methods based on
imputation and pseudo-labeling (e.g.,~\cite{chakrabortty2018,
  wang2023pseudo, lee2013pseudo}), the EM algorithm
(e.g.,~\cite{mclachlan2008algorithm, wu1983convergence, BalWaiYu17}),
and other semi-parametric approaches
(e.g.,~\cite{robins1994estimation, robins1995semiparametric}).  There
is also a considerable body of work on the inferential aspects of the
semi-supervised set-up
(e.g.,~\cite{10.1214/18-AOS1756,zhang21semisup}).

The framework of this paper differs from the standard semi-supervised
set-up via the introduction of helper covariates.  This notion
connects with, and establishes links between other previously
unrelated lines of research.  As described in the previous section, it
is natural to view surrogate responses collected from a pre-fit
machine learning model as a form of helper covariate.  This connects
with the machine learning literature on weak supervision
(e.g.,~\cite{ratner2016data, ratner2017snorkel,
  robinson2020strength}), as well as a related literature focusing on
robustness to noisy labels (e.g.,~\cite{natarajan2013learning,
  song2022learning}).  In causal inference, researchers have studied
the use of surrogate indices (e.g.,~\cite{athey2019surrogate,
  kallus2020role, hou2023surrogate}), which can be exploited in
estimating a treatment effect.  These indices can be understood as
particular types of helper covariates, and the methodology in the
paper~\cite{hou2023surrogate} is related to our general three-stage
approach, albeit in a specific parametric setting.  In the context of
transfer learning, Bastani~\cite{bastani2021predicting} also studies a
related method, in which both tasks are parameterized by
high-dimensional linear models.  Finally, related in spirit (but
distinct in goals) is the use of a pre-fit machine learning model to
construct sharper confidence
intervals~\cite{angelopoulos2023prediction, zrnic2024cross}.


\paragraph{Notation:} We collect here notation that is used throughout
the paper.  Given a collection of samples $X_i \sim \Prob$ for $i = 1,
\ldots, \numobs$,, we define $\| g \|_{\tnumobs}^2 \defn
\frac{1}{\tnumobs} \sum_{i=1}^\tnumobs g^2(\Covariate_i)$ and $\| g
\|_2^2 \defn \int g^2(\Covariate) \, d\PP(X)$.  Note that $\|g\|_2^2 =
\EE[g^2(X)]$ whenever $g$ is deterministic, or more generally, even if
$g$ is random, as long as it is independent of $X$.  When we write
$\Exs[\auxhat^2(X)]$ for a random function independent of $X$, we are
conditioning on $\auxhat$.  Similarly, we define $\PP_{\tnumobs} g
\defn \empsum{i} g(\Covariate_i)$ and $\PP g \defn \int g(\Covariate)
\, d\PP(\Covariate)$, with similar comments regarding the meaning of $\PP
\auxhat$ for a random function $\auxhat$.  Using this notation, we
have $\| g \|_{\tnumobs}^2 = \PP_{\tnumobs} g^2$ and $\| g\|_2^2 = \PP
g^2$. It is also convenient to consider the above norms over subsets
of the data $\dataset \subset \{\Covariate_i\}_{i=1}^\numobs$, for
which purpose we define $\| g \|_\dataset^2 \defn \frac{1}{|\dataset|}
\sum_{\Covariate \in \dataset} g^2(\Covariate)$.

\paragraph{Organization:}  The remainder of the paper is organized as
follows.  In \Cref{sec:main-results}, we formalize the problem of
prediction with helper covariates, and then describe the
\ERIC~procedure analyzed in this paper.  Later subsections are devoted
to theoretical guarantees for squared-loss (cf.  \Cref{thm:sqloss}
in~\Cref{sec:sqloss-results}) and more general loss functions
(cf. \Cref{thm:main} in~\Cref{SecGeneralLoss}).
In~\Cref{SecEnsemble,SecBinary}, respectively, we explore some
qualitative predictions of our theory for least-squares problems and
binary classification, respectively.  \Cref{sec:empirical} is devoted
to exploration of four different applications in which helper
covariates naturally arise, and of the improvements
possible by the \ERIC~method.  We conclude with a summary as well as
future directions in~\Cref{sec:conclusion}. All of our proofs are
deferred to the appendices.


\section{Surrogate-aided prediction via the \ERIC~method}
\label{sec:main-results}

We begin with a formalization of the problem of surrogate-aided
prediction in~\Cref{sec:problem}.  In~\Cref{SecPast}, we introduce our
proposed meta-procedure, referred to as ``prediction aided surrogate
training'', or the \ERIC~method for short.  \Cref{sec:sqloss-results}
is devoted to a non-asymptotic and explicit guarantee
(\Cref{thm:sqloss}) on its behavior for the squared-loss error; this
theory provides a number of qualitative insights that are explored
in~\Cref{SecEnsemble}.  In~\Cref{SecGeneralLoss}, we provide a more
general result (\Cref{thm:main}) that covers a broad class of loss
functions.  We discuss concrete insights for binary classification
in~\Cref{SecBinary}.


\subsection{Prediction with helper covariates}
\label{sec:problem}

We begin by formalizing the problem of prediction with helper
covariates.  Let $\Response \in \Yspace \subseteq \real$ be a response
variable of interest, and let $\Covariate \in \Xspace$ be a covariate
vector, where the pair $(\Covariate, \Response)$ are distributed
according to some unknown joint distribution $\PP_{\Covariate,
  \Response}$ over the space $\Xspace \times \Yspace$.  Our goal is to
estimate a good predictor, meaning a function $\covariate \mapsto
f(\covariate)$ such that $f(\covariate)$ is ``close'' to the
associated response $\response$.  We measure the quality of a given
prediction via some loss function $\loss: \real \times \real
\rightarrow \real$---so that $\loss(f(\covariate), y)$ is our accuracy
measure--- and our goal to find a function $\fstar$ that minimizes the
population risk
\begin{align}
\label{EqnPopulationRisk}
\fstar & = \arg \min_{f \in \Fclass} \EE_{\Covariate, \Response}
\loss(f(\Covariate), \Response),
\end{align}
where $\Fclass$ is a given function class, and $\EE_{\Covariate,
  \Response}$ denotes expectation under the joint $\PP_{\Covariate,
  \Response}$. For certain loss functions---among them the
least-squares loss $\loss(\yhat, y) = (\yhat - y)^2$---the optimal
prediction function (assuming that $\Fclass$ is sufficiently rich) is
given by the conditional mean $\tarstar(x) \defn \Exs[Y \mid X = x]$,
but our theory allows for other settings as well.  Throughout this
paper, we refer to $\tarstar$ as the \emph{target function}, since our
ultimate goal is to obtain an accurate estimate $\tarhat$ of it.

In the standard set-up of prediction via empirical risk minimization,
the statistician is given a collection of pairs $(\Covariate_i,
\Response_i)$ drawn in an i.i.d. manner from the unknown distribution
$\Prob_{\Covariate, \Response}$.  In the helper-aided extension here,
we augment this standard set-up as follows.  In addition to the pair
$(\Covariate, \Response)$, we introduce the helper covariate
$\Surrogate$, and assume that is a unknown joint
distribution\footnote{This three-way joint marginalizes down to
$\Prob_{\Covariate, \Response}$ that defines the population
risk~\eqref{EqnPopulationRisk}.} $\Prob_{\Covariate, \Surrogate,
  \Response}$.  Our goal remains to estimate the function $\fstar$
defined by the population risk~\eqref{EqnPopulationRisk}, and in order
to do so, we are given access to a total of $\numobs$ independent
samples.  The full sample size is partitioned as $\numobs = \lnumobs +
\unumobs$, where $\lnumobs$ and $\unumobs$ correspond (respectively)
to the number of ``labeled'' and ``unlabeled'' samples, respectively.
The \emph{labeled and unlabeled datasets} take the form
\begin{align}
\label{EqnDatasets}  
\LabelSet = \{(\Covariate_i, \Surrogate_i, \Response_i)
\}_{i=1}^\lnumobs \quad \mbox{and} \quad \UnlabelSet =
\{(\Covariate_{\lnumobs + i}, \Surrogate_{\lnumobs + i} ) \big
\}_{i=1}^\unumobs,
\end{align}
where $(\Covariate_i, \Surrogate_i, \Response_i)$ are drawn
i.i.d. from $\Prob_{\Covariate, \Surrogate, \Response}$ for $i = 1,
\ldots, \lnumobs$, and $(\Covariate_{\lnumobs + i},
\Surrogate_{\lnumobs + i})$ are drawn i.i.d. from $\Prob_{\Covariate,
  \Surrogate}$ for $i = 1, \ldots, \unumobs$.  Given the independent
sampling model described here, in the language of missing data, our
assumption is that the responses are missing completely-at-random.  We
address possible relaxations of this condition in the discussion.

Summarizing then, given access to the datasets $\LabelSet$ and
$\UnlabelSet$, our goal is to estimate a predictor $\covariate \mapsto
\fhat(\covariate)$ that is close to the population optimum $\fstar$
from equation~\eqref{EqnPopulationRisk}.  While the final predictor
cannot depend on the helper covariates $\Surrogate$, they can be
exploited at an intermediate phase of training, as we describe next.


\subsection{The \ERIC~method}
\label{SecPast}

In this paper, we analyze a three-stage procedure that takes as input
the labeled $\LabelSet$ and unlabeled $\UnlabelSet$ datasets
(cf. equation~\eqref{EqnDatasets}), and returns a predictor $\fhat$
as output.  In the first step, we use the labeled dataset $\Labelset$
to construct a pseudo-response estimator $\auxhat$.  In the second
step, we use $\auxhat$ to impute responses for the unlabeled dataset
$\Unlabelset$, and in the third step, we fit the final predictor
$\fhat$ using the pseudo-labeled set $\FullDataset$.
Algorithm~\ref{AlgPast} provides a formal specification of these three
steps; we describe it as ``prediction aided by surrogate training'',
and refer to it as the \ERIC~procedure for short.

\begin{algorithm}[h]
\caption{Prediction Aided by Surrogate Training (\ERIC)}
\begin{algorithmic}[1]
\STATE \texttt{Inputs:} (i) Datasets $\LabelSet = \{(\Covariate_i,
\Surrogate_i, \Response_i)\}_{i=1}^{\lnumobs}$ and $\UnlabelSet =
\{(\Covariate_i, \Surrogate_i) \}_{i=\lnumobs+1}^\tnumobs$. (ii)
Procedure $\Proc$ for estimating pseudo-response function
$\auxhat$. (iii) Loss function $\Loss$ and function class $\TarClass$
for estimating $\tarstar$.
\vspace{6pt}

\STATE Construct pseudo-response function $\auxhat$ by applying
procedure $\Proc$ to the labeled dataset $\LabelSet$.

\STATE Use $\auxhat$ to construct the full dataset based on
pseudo-responses
\begin{align*}
\Ytil_i = \begin{cases} \Response_i & \text{for} \quad i = 1, \ldots,
  \lnumobs \\ \auxhat(\Covariate_i, \Surrogate_i) & \text{for} \quad i
  = \lnumobs+1, \ldots, \tnumobs. \end{cases}
\end{align*}

\STATE Using the dataset $\DatasetTil = \{(\Covariate_i, \Ytil_i)
\}_{i=1}^\numobs$, compute the estimate
\begin{align}
  \label{EqnDefnTarHat}
\tarhat \in \arg \min_{\tarplain \in \TarClass} \Big\{
\frac{1}{\tnumobs} \sum_{i=1}^\tnumobs \loss \big(
\tarplain(\Covariate_i), \Ytil_i \big) \Big\}.
\end{align}
\end{algorithmic}
\label{AlgPast}
\end{algorithm}

We emphasize that the \ERIC~method should be viewed as a
\emph{meta-procedure}, since we obtain different instantiations
depending on:
\begin{cdesc}
\item[Auxiliary fit:] the choice of procedure $\Proc$ for constructing
  the pseudo-response estimator $\auxhat$.
\item[Loss function:] The loss function used in
  step~\eqref{EqnDefnTarHat}, and
 \item[Function class:] The choice of function class $\Fclass$ for
   estimating $\tarhat$.
\end{cdesc}

We measure the quality of the final predictor $\fhat$ via its closeness
to the optimal function $\fstar$ defined by equation~\eqref{EqnPopulationRisk}.
In our
analysis, we remain agnostic to the choice of the procedure $\Proc$
used to compute $\auxhat$, and state a guarantee on the quality of the
final output $\fhat$---in particular, an upper bound on $\|\fhat -
\fstar\|_2$.  This guarantee involves an overhead term depending on
the deviation between $\auxhat$ and an ideal function $\auxstar$.
When using the least-squares loss function, we have
$\fstar(\covariate) = \Exs[\Response \mid \Covariate = \covariate]$,
and this \emph{ideal proxy} is given by
\begin{align}
\label{EqnIdealProxy}
\auxstar(\covariate, \surrogate) & \defn \Exs \big[ \Response \mid
  (\Covariate, \Surrogate) = (\covariate, \surrogate) \big].
\end{align}
This function $\auxstar$ is also the ideal proxy when using losses
based on generalized linear models (see~\Cref{sec:loss-functions} for
details), of which the least-squares loss is a special case.

Our results depend on the accuracy of $\auxhat$ as an estimate of
$\auxstar$.  For this reason, even when the response space is
discrete---say $\Yspace = \{0,1 \}$ in the case of binary
classification---our theory reveals that superior performance can be
obtained by using ``soft label'' information.  \Cref{thm:main} and the
surrounding discussion to follow makes this intuition precise.

\medskip


\subsection{Guarantees for square loss}
\label{sec:sqloss-results}

In this section, we state a guarantee for the \ERIC~procedure when
$\fhat$ is estimated using least-squares---that is, based on the loss
function $\loss(\yhat, y) = (\yhat - y)^2$.  At a high-level, this
guarantee (and the subsequent ones in~\Cref{thm:main} to follow)
involve three terms:
\begin{cdesc}
\item[Oracle accuracy:] The quantity $\radcrit> 0$, to be defined in
  equation~\eqref{EqnDefnRadcrit}, corresponds to the the rate at
  which it is possible to estimate $\tarstar$ using a fully labeled
  dataset of size $\numobs$.
\item[Pseudo-response accuracy:] Accuracy of pseudo-response estimator
  $\auxhat$ relative to ideal proxy $\auxstar$ from
  equation~\eqref{EqnIdealProxy}.  Ideally, it involves the
  $\Surrogate$-smoothed version $\ftil$ of the estimate $\auxhat$
  (cf. equation~\eqref{EqnDefnFtil}).
\item[Higher-order fluctuations:] For a given error probability
  $\delta \in (0,1)$, there is a term $\HighOrder_\numobs(\delta)$
  that controls the probabilistic fluctuations
  (cf. equation~\eqref{EqnDefnHigh}).
\end{cdesc}
\noindent We now introduce each of these terms in turn.

\paragraph{Oracle accuracy:} Beginning
with the oracle accuracy $\radcrit$, it is defined via a fixed point
relation involving the localized Rademacher complexity.  Quantities of
this type are measures of the ``complexity'' of a given function
class, and are well-known to control the accuracy of procedures based
on empirical risk minimization.  In particular, let
$\{\rade_i\}_{i=1}^\numobs$ denote an i.i.d. sequence of Rademacher
variables.\footnote{That is, we have $\rade_i \in \{-1, +1 \}$
equi-probably.}  For any radius $\myrad > 0$, we define
\begin{subequations}
  \begin{align}
\label{EqnDefnLocalRad}    
\radcomp(\myrad) & \defn \Exs_{\rade, \Covariate} \Biggr[ \sup_{
    \substack{\tarplain \in \Fclass \\ \|\tarplain - \tarstar\|_2 \leq
      \myrad}} \Big| \frac{1}{\numobs} \sum_{i=1}^\numobs \rade_i
  \big(\tarplain(\Covariate_i) - \tarstar(\Covariate_i) \big) \Big|
  \Biggr].
\end{align}
We then define the \emph{oracle accuracy} as
\begin{align}
\label{EqnDefnRadcrit}  
\radcrit & \defn \arg \min_{\myrad > 0} \Big \{ \frac{\myrad}{16}
\geq \frac{\radcomp(\myrad)}{\myrad} \Big \}.
\end{align}
\end{subequations}
To explain the significance of this quantity for our theory, suppose
that we were able to compute an estimate $\tardagger$, of the form in
equation~\eqref{EqnDefnTarHat}, but with the pseudo-responses
$\Ytil_i$ all replaced by genuine responses $\Response_i$.  Under
quite general conditions, it can be shown that this oracle estimate
$\tardagger$ would have error of the form $\|\tardagger - \tarstar\|_2
\asymp \radcrit$.  For this reason, the term $\radcrit$ serves as a
proxy for the error incurred by empirical risk minimization on a fully
labeled dataset.  As some concrete examples, we have:

\begin{carlist}
\item For linear regression in dimension $\usedim$, we have
  $\radcrit \asymp \sqrt{\usedim/\numobs}$.
\item For linear regression in dimension $\usedim$ with sparsity
  $\kdim$, we have $\radcrit \asymp \sqrt{\frac{\kdim \log
    \usedim}{\numobs}}$.
\item For Lipschitz regression in dimension $\usedim = 1$, we have
  $\radcrit \asymp \numobs^{-1/3}$.
\end{carlist}
\smallskip

\noindent See Chapter 13 in the book~\cite{dinosaur2019} for these
results, and other examples. \\

\smallskip

Our proofs assume that the oracle accuracy, as defined by the
Rademacher complexity, satisfies the following \emph{monotonicity
condition:} for any pair of sample sizes with $\numobs_1 < \numobs_2$,
we have
\begin{align}
  \label{EqnRadeMonotone}
\sqrt{\numobs_1} \; r_{\numobs_1} \; \leq \; c \sqrt{\numobs_2} \; r_{\numobs_2}.
\end{align}
Note that this condition holds whenever $\radcrit$ is lower bounded
by $1/\sqrt{\numobs}$, which holds for most function classes of
interest, as in the examples given above. Up to scaling factors in our guarantee, we can assume without loss of generality that $c = 1$.

\paragraph{Quality of pseudo-responses:}  The second term in our
analysis reflects the quality of the pseudo-response estimator
$\auxhat$ as a surrogate to the ideal version $\auxstar$ from
equation~\eqref{EqnIdealProxy}.  Since $\auxhat$ is applied to impute
responses on on the $\unumobs$ unlabeled samples, the most
straightforward measure of accuracy is the squared empirical norm
\begin{subequations}
\begin{align}
\label{EqnNaiveMeasure}  
\|\auxhat - \auxstar\|^2_\unumobs & \defn \frac{1}{\unumobs} \sum_{i
  \in \Unlabelset} \big( \auxhat(\Covariate_i, \Surrogate_i) -
\auxstar(\Covariate_i, \Surrogate_i) \big)^2
\end{align}
defined by samples in the unlabeled set.  However, most of our results
measure the accuracy in terms of the same squared semi-norm applied to
the \emph{$\Surrogate$-smoothed functions} $\ftil$ and the true
regression function $\fstar$---viz.
\begin{align}
\label{EqnDefnFtil}  
\ftil(\covariate) \defn \Exs[\auxhat(\Covariate, \Surrogate) \mid
  \Covariate = \covariate], \quad \mbox{and} \quad \fstar(\covariate)
\equiv \Exs[\auxstar(\Covariate, \Surrogate) \mid \Covariate =
  \covariate].
\end{align}
\end{subequations}
As can be seen via Jensen's inequality, the squared semi-norm $\|\ftil
- \fstar\|_\unumobs^2$ is always smaller than---and can be
substantially so---than the quantity $\|\auxhat -
\auxstar\|_\unumobs^2$ defined in equation~\eqref{EqnNaiveMeasure}.

\paragraph{Higher-order term:}  Finally, given some error probability
$\delta \in (0, 1)$, we prove results that hold with probability at
least $1 - \delta$.  Obtaining high probability guarantees of this
type requires introducing a higher-order term
$\HighOrder_\numobs(\delta)$ depends on the boundedness and/or tail
behavior assumed.  Since our main focus is not obtaining the most
general results, we impose relatively stringent assumptions that allow
for a streamlined analysis.  In particular, we assume throughout that
the function $\Fclass$ is \emph{uniformly bounded}; by rescaling as
needed, we can assume that $\|f\|_\infty \leq 1$ for all $f \in
\Fclass$, and that the same holds for $\auxstar$.  Moreover, recalling
the definition~\eqref{EqnIdealProxy} of $\auxstar$, we assume that the
zero-mean noise variables $\Response_i - \auxstar(\Covariate_i,
\Surrogate_i)$ is \emph{$\sigbound$-bounded}, meaning that it takes
values in the interval $[-\sigbound, \sigbound]$ for some $\sigbound >
0$.  To be clear, both of these conditions could be relaxed by making
use of more sophisticated tail bounds for empirical
processes~\cite{adamczak2008unbounded}, but this is not the main focus
of our work here. \\

In terms of these assumptions, the higher-order term
in~\Cref{thm:sqloss} takes the form
\begin{align}
\label{EqnDefnHigh} 
\HighOrder_\numobs(\pardelta) \defn \max\{20, 10\sigma\}
\sqrt{\frac{2\log(4\Logfun(\radcrit)/\pardelta)}{\numobs}} +
\max\{640, 80\sigbound\} \cdot
\frac{\log(4\Logfun(\radcrit)/\pardelta)}{\radcrit \numobs}
\end{align}
where $\Logfun(\radcrit) \defn
\log_2\big(\tfrac{4\sigbound}{\radcrit}\big)$.  While we provided
explicit constants here for concreteness, we made no attempt in our
analysis to obtain sharp ones.

\bigskip

With these three pieces in place, we are now ready to state our main
guarantees for the \ERIC~procedure based on least-squares loss.
Recapitulating for clarity, it involves the \emph{oracle accuracy}
from equation~\eqref{EqnDefnRadcrit}, the empirical norm
$\|\cdot\|_\unumobs$ and \emph{$\Surrogate$-smoothed response
estimator} $\ftil$ from equations~\eqref{EqnNaiveMeasure}
and~\eqref{EqnDefnFtil}, respectively, and the higher-order
term~\eqref{EqnDefnHigh}.

\begin{theorem}
\label{thm:sqloss}
For a bounded function class $\Fclass$ and $\sigbound$-bounded noise,
consider the \ERIC~procedure implemented with an arbitrary first-stage
predictor $\auxhat$, and the $\Surrogate$-smoothed version $\ftil$ of
$\auxhat$. Then for any $\pardelta \in (0, 1)$, the final output
$\fhat$ of the \ERIC~procedure has prediction error at most
\begin{align}
\label{EqnSqLossBound}
\|\tarhat - \tarstar \|_2 \leq \underbrace{(11 + 10 \sigbound)
  \radcrit}_{\mbox{Oracle risk}} + \underbrace{3 \EmpNormUnlab{\ftil -
    \tarstar}}_{\mbox{Pseudo-response defect}} + \underbrace{2
  \HighOrder_\numobs(\pardelta)}_{\mbox{Higher-order term}}
\end{align}
\mbox{with probability at least $1 - \pardelta$.}
\end{theorem}
\noindent See~\Cref{sec:proof-sqloss} for a proof of this result. \\

As previously described, the bound~\eqref{EqnSqLossBound} consists of
three terms: one involving the critical radius $\radcrit$, the second
involving the difference $\ftil - \tarstar$, and the third term
$\HighOrder_\numobs(\pardelta)$ corresponding to probabilistic
fluctuations.

\paragraph{When can we achieve the oracle risk?}
It is interesting to ask under what conditions the first term is
dominant; it corresponds to a type of oracle risk that could be
achieved if we had access to $\numobs$ labeled samples.  We claim that
the term $\HighOrder_\numobs(\delta)$ is typically negligible with
respect to $\radcrit$.  Indeed, apart from scalar parametric problems,
we have $\radcrit \succsim \numobs^{-1/2}$, in which case inspection
of the definition~\eqref{EqnDefnHigh} shows that we have
$\HighOrder_\numobs(\pardelta) \ll \radcrit$.

Thus, we are left to consider the estimation error term.  It is an
important fact---and one that requires delicacy in our proof---that
this term depends \emph{not} on the difference $\auxhat - \auxstar$,
but rather on the difference $\ftil - \tarstar$.  (Here we recall that
$\ftil(\covariate) = \Exs[\auxhat(\Covariate, \Surrogate) \mid
  \Covariate = \covariate]$ is the $\Surrogate$-smoothed version of
our auxiliary predictor $\auxhat$.)  By Jensen's inequality, we always
have
\begin{align*}
\EmpNormUnlab{\ftil - \tarstar} & \leq \EmpNormUnlab{\auxhat -
  \auxstar},
\end{align*}
and this difference can be substantial in many settings.  As one
simple but extreme example, suppose that we were given access to the
function $\auxhat(\Covariate_i, \Surrogate_i) = \Response_i$, so that
our pseudo-responses were actually correct.  In this case, we have
$\EmpNormUnlab{\auxhat - \auxstar} \neq 0$, but $\ftil = \fstar$.

Otherwise, since the function $\auxhat$ (and hence $\ftil$) is
obtained by training on the labeled dataset $\LabelSet$, the rate at
which the error $\EmpNormUnlab{\ftil - \tarstar}$ decays will depend
on $\lnumobs = |\LabelSet|$.  Thus, in order for this error to be of
lower order relative to the oracle risk $\radcrit$---which depends on
the full sample size $\numobs = \unumobs + \lnumobs$---it should be
the case that the problem of predicting $\Response$ based on the pair
$(\Covariate, \Surrogate)$ should be ``easier'' than that of
predicting $\Response$ based on $\Covariate$ alone.  As discussed in
the introduction, many of our motivating examples---among them
forecasting with future information as surrogates, or prediction with
noisy labels as surrogates---are likely to have this property.  We
explore this issue in a more depth via some synthetic ensembles in the
following subsection.


\subsection{Intuition from some simple ensembles}
\label{SecEnsemble}

The guarantee from~\Cref{thm:sqloss} makes specific predictions about
the quantitative trade-offs inherent to making use of helper
covariates.  In this section, we describe some simple partially linear
ensembles of problems---chosen for illustrative purposes---that
highlight both the sharpness of our theoretical predictions, as well
as the insight that~\Cref{thm:sqloss} provides into the behavior of
the \ERIC~procedure.  We begin in~\Cref{SecEnsembleOne} with a very
simple ensemble that titrates the helpfulness of the helper covariate,
revealing the importance of the signal-to-noise ratio associated with
predicting the ideal proxy $\auxstar$.  In~\Cref{SecEnsembleTwo}, we
turn to a more challenging ensemble in which the helper covariate can
actually be harmful.

\subsubsection{How useful is the helper covariate?}
\label{SecEnsembleOne}

We begin by considering a simple ensemble of problems in which the
utility of the helper covariate can be controlled by a single
parameter $\lambda \in [0, 1]$.  In particular, suppose that we
generate responses $\Response \in \real$ according to the model
\begin{align}
\label{EqnEnsembleOne}  
Y & = \tarstar(X) + \lambda \Surrogate + (1 - \lambda) \, \gnoise,
\end{align}
where the covariate $\Covariate$ is independent of the helper
$\Surrogate$ and noise variable $\gnoise$, and $\fstar$ is some
function of the standard covariate only. For concreteness, we choose
$\Surrogate$ and $\gnoise$ as independent random variables, each with
a $\Normal(0, \sigma^2)$ distribution.  With this set-up, we have the
equivalences \mbox{$\Exs[Y \mid \Covariate = \covariate] =
  \tarstar(\covariate)$,} and \mbox{$\auxstar(\covariate, \surrogate)
  = \tarstar(\covariate) + \lambda \surrogate$.}  Moreover, observe
that
\begin{align}
\label{EqnHelperVariance}  
  \var \big( \Response \mid \Covariate = \covariate \big) = \sigma^2,
  \quad \mbox{whereas} \quad \var \big( \Response \mid (\Covariate,
  \Surrogate) = (\covariate, \surrogate) \big) = (1 - \lambda)^2 \,
  \sigma^2.
\end{align}
Consequently, the utility of the helper covariate $\Surrogate$ is
controlled by $\lambda$.  When $\lambda = 1$, the response $\Response$
is known deterministically once $\Covariate$ is observed in
conjunction with $\Surrogate$; at the other extreme ($\lambda = 0$),
the helper covariate $\Surrogate$ is independent of the response
$\Response$, and so not at all helpful!

Let us perform some additional calculations to understand some
predictions of~\Cref{thm:sqloss} for this particular ensemble.
Suppose---for the sake of concreteness and simplicity---that $\fstar$
is a polynomial function of some degree; we can then write it in the
form $\fstar(\covariate) = \inprod{\betastar}{\Feat(\covariate)}$ for
some feature vector $\Feat(\covariate) \in \real^\usedima$, and some
weight vector $\betastar \in \real^\usedima$.  In this setting, some
standard calculations (see~\Cref{AppEnsemble}) show
that~\Cref{thm:sqloss}, with $\auxhat$ estimated by an appropriate
linear regression, leads to a guarantee of the form
\begin{align}
\label{EqnGuaranteeOne}  
\|\tarhat - \tarstar\|_2 \lesssim \sigma
\sqrt{\frac{\usedima}{\numobs}} + \sigma (1 - \lambda) \,
\sqrt{\frac{\usedima}{\lnumobs}},
\end{align}
where $\numobs$ is the total sample size, and $\lnumobs < \numobs$ is
the number of labeled samples.  Recall that in this case, we have
$\radcrit \asymp \sqrt{\usedima/\numobs}$, corresponding to the oracle
accuracy achievable with $\numobs$ samples.  At the extreme $\lambda =
0$, the helper covariates are independent of the response, in which
case we should expect an overall guarantee scaling as
$\sqrt{\usedima/\lnumobs}$. (Indeed, the unlabeled dataset
$\UnlabelSet$ is useless from the information-theoretic perspective.)
But more optimistically, as long as $\lambda$ is sufficiently large
\footnote{Concretely, whenever $\lambda \in \big[ 1 - c
  \sqrt{\lnumobs/\numobs}, \: 1 \big]$ for some constant $c > 0$, the
right-hand side of the bound~\eqref{EqnGuaranteeOne} scales as $(1 +
c) \sigma \sqrt{\usedima/\numobs}$.}  relative to the ratio
$\frac{\lnumobs}{\numobs} \in (0, 1]$ of labeled samples, we achieve a
  guarantee proportion to the oracle accuracy
  $\sqrt{\usedima/\numobs}$.

\Cref{FigNumSim}(a) provides a graphical illustration of the
correspondence between the theoretically predicted
bound~\eqref{EqnGuaranteeOne}, and empirical performance for synthetic
data.  In particular, we generated responses of the
type~\eqref{EqnEnsembleOne} by drawing $\Covariate \sim
\text{Unif}([0,1]^5)$; choosing the feature mapping
$\Feat(\covariate)$ so as to produce all polynomial interactions of
order up to $3$ for the entries in $\covariate$, thereby yielding a
total dimension $\usedima = 56$; and choosing
the helper $\Surrogate$ and noise variable $\gnoise$ as $\Normal(0,
\sigma^2)$, with both independent of each other and the covariate.

We generated a dataset of total size $\numobs = 1000$, and partitioned
it into a labeled portion of size $|\LabelSet| = \lnumobs = 100$, and
unlabeled portion of size $|\UnlabelSet| = \unumobs = 900$. We then
applied the \ERIC~procedure by first computing an estimate $\auxhat$
via the linear regression described in~\Cref{AppEnsembleOne}, and
using it to generate the pseudo-responses.  We then estimated the
polynomial function $\fstar$ by polynomial regression, and compared
the performance of the resulting estimator $\fhat$ to that of the
naive method which makes use of the labeled dataset $\LabelSet$
only. Panel (a) plots of the root mean-squared error (RMSE) $\|\fhat -
\fstar\|_2$ versus the parameter $\lambda \in [0,1]$.  Consistent with
the theoretical prediction~\eqref{EqnGuaranteeOne}, we that the RMSE
decreases linearly as $\lambda$ increases; moreover, it always lies
below the error of the naive procedure, and it approaches the
performance of the oracle procedure (given access to a fully labeled
dataset of size $\numobs$) as $\lambda \rightarrow 1$.

\begin{figure}[h]
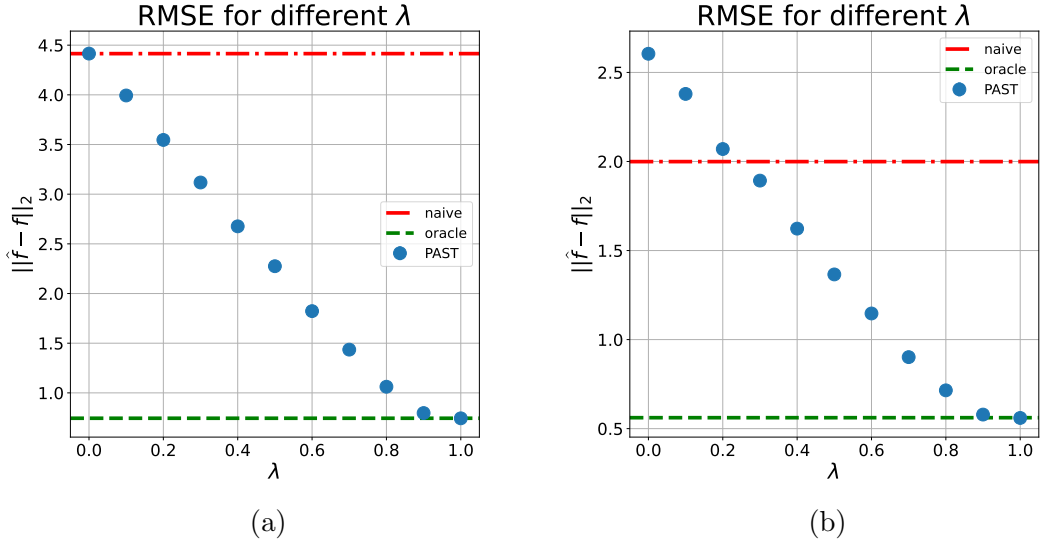

  \begin{center}
    \begin{tabular}{cc}
      \widgraph{0.45\textwidth}{\figdir/nonlinear_simulations} &
      \widgraph{0.45\textwidth}{\figdir/nonlinear_simulations_v2} \\
      (a) & (b)
    \end{tabular}
    \caption{Comparisons of theoretical predictions to numerical
      behavior.  Blue dots correspond to the root mean-squared error
      (RMSE) of the \ERIC~procedure for a given ensemble.  The green
      dotted line corresponds to the RMSE of an oracle given a fully
      labeled dataset of size $\numobs$, whereas the black dotted line
      corresponds to the performance of a method that uses only the
      $\lnumobs$ labeled samples.  (a) Plots for
      ensemble~\eqref{EqnEnsembleOne}: exploiting the helper covariate
      is always beneficial. Consistent with the theoretical
      prediction~\eqref{EqnGuaranteeOne}, the \ERIC~procedure always
      has lower RMSE than the naive method, and it converges to oracle
      performance as $\lambda \rightarrow 1$.  (b) Plots for
      ensemble~\eqref{EqnEnsembleTwo}: helper covariate can be
      harmful. Consistent with the theoretical
      prediction~\eqref{EqnGuaranteeTwo}, there is now an interval of
      $\lambda$ over which the \ERIC~procedure has larger RMSE than
      the naive method; its performance still converges to oracle
      performance as $\lambda \rightarrow 1$.}
    \label{FigNumSim}
  \end{center}
\end{figure}


\subsubsection{When can a ``helper'' be harmful?}
\label{SecEnsembleTwo}

The preceding example was rather benign, in the sense that the
strategy of introducing the helper covariate---while helpful when
$\lambda$ was sufficiently large---was \emph{never} harmful.  In this
section, we exploit the insight afforded by~\Cref{thm:sqloss} to
construct a more challenging ensemble in which there is a transition
between helpful and harmful.  We consider an extension of our previous
ensemble~\eqref{EqnEnsembleOne}, in which the helper covariate is a
higher dimensional vector. It can be partitioned as $\Surrogate = (U,
V)$ for some pair $U \in \real$ and $V \in \real^{\usedimb}$, and the
response is generated as
\begin{align}
\label{EqnEnsembleTwo}
\Response & = \fstar(\Covariate) + \lambda U + \inprod{\alphastar}{V}
+ (1 - \lambda) \, \gnoise,
\end{align}
where $\alphastar \in \real^{\usedimb}$ is some unknown vector.  For
this ensemble, the ideal proxy is given by $\auxstar(\covariate,
\surrogate) = \fstar(\covariate) + \lambda u +
\inprod{\alphastar}{v}$, where $\surrogate = (u, v)$.

As with our earlier ensemble~\eqref{EqnEnsembleOne}, the conditional
variances of the response $\Response$ obeys the
relation~\eqref{EqnHelperVariance}, so that the benefit of the helper
covariate depends on the parameter $\lambda \in [0, 1]$.  In contrast,
however, our new ensemble~\eqref{EqnEnsembleTwo} involves a
potentially high-dimensional helper covariate $\Surrogate = (U, V)$;
it is only the scalar variable $U$ that is potentially helpful, depending
on the value of $\lambda$, whereas the vector $V \in \real^{\usedimb}$
is simply a ``nuisance quantity''.  As we will see, since the
partition of $\Surrogate$ into $U$ and $V$ is unknown to us,
exploiting the helper $W$ can actually be harmful; roughly speaking,
it is harmful whenever the utility of information in $U$ is dominated
by the noise introduced by $V$.  Our theory reveals how this effect
becomes prominent when the nuisance dimension $\usedimb$ is large, and
$\lambda$ is sufficiently small.

We make this intuition precise by returning to the case when
$\fstar(\covariate) = \inprod{\betastar}{\Feat(\covariate)}$ for some
unknown $\betastar \in \real^{\usedima}$ and known feature vector
$\Feat(\covariate) \in \real^{\usedima}$. As shown
in~\Cref{AppEnsembleTwo}, using non-asymptotic bounds from the
paper~\cite{xia2024iv}, when the ideal proxy $\auxstar$ is
estimated via a linear regression, then it can be shown that the final
output $\fhat$ of the \ERIC~procedure has RMSE bounded as
\begin{align}
\label{EqnGuaranteeTwo}
\|\tarhat - \tarstar \|_2 \lesssim \sigma
\sqrt{\frac{\usedima}{\numobs}} + \sigma (1 - \lambda) \Big(
\sqrt{\frac{\usedima}{\lnumobs}} + \frac{\usedima +
  \usedimb}{\lnumobs} \Big).
\end{align}
Thus, we see that---in addition to $\lambda \in [0,1]$---the dimension
$\usedimb$ plays a non-trivial role.  Concretely, consider a problem
for which $\usedimb \approx \frac{\lnumobs}{2}$.  Then the rightmost
term in the bound~\eqref{EqnGuaranteeTwo} contains a term proportional
$\sigma (1 - \lambda)$.  Thus, assuming that the upper bound from our
theory is sharp, it predicts that for $\lambda$ sufficiently close to
zero, the \ERIC~procedure could have performance inferior to the naive
method.

\Cref{FigNumSim}(b) provides confirmation that this prediction is
accurate.  As in~\Cref{SecEnsembleTwo}, we simulated $\numobs = 1000$
samples, partitioned as $|\LabelSet| = \lnumobs = 100$ and
$|\UnlabelSet| = \unumobs = 900$.  In the response
model~\eqref{EqnEnsembleTwo}, we chose $\fstar$ in a similar manner as
above, a polynomial function in dimension $5$ with maximum degree two,
(corresponding to a dimension $\usedima = 21$ in the linearized
formulation); and helper covariate $\Surrogate = (U, V) \in \real
\times \real^{\usedimb}$ with dimension $\usedimb = 50$.  We made
independent draws of the standard covariate $\Covariate \sim
\text{Unif}([0,1]^5)$; first helper component $U \sim \Normal(0,
\sigma^2)$; second helper component $V \sim \text{Unif}([-1,1]^{50})$;
and additive noise variable \mbox{$\gnoise \sim \mathcal{N}(0,
  \sigma^2)$.}  Panel (b) in the figure shows the correspondence
between the theoretically predicted bound~\eqref{EqnGuaranteeTwo}, and
the resulting empirical performance.  Consistent with the prediction,
we again see that the RMSE of the \ERIC~procedure decreases linearly
in $\lambda$, with the important difference that for $\lambda$
sufficiently small (roughly, in the interval $[0, 0.4])$), its
performance is actually worse than the naive procedure.  In this
regime, the utility of component $U$ in the helper covariate is
overwhelmed by the noise introduced by the higher-dimensional
\mbox{$V$-component.}  Overall, we conclude that with high-dimensional
helper covariates, some care is required in when the \ERIC~procedure
should be applied; we offer some practical guidance in the discussion
section.


\subsection{Guarantees for general losses}
\label{SecGeneralLoss}

Thus far, we have provided results that apply to the quadratic loss
function \mbox{$\Loss(\yhat, y) = (\yhat - y)^2$.} In this section, we
turn to an analysis of a more general class of loss functions.  Recall
that the target function $\fstar$ is defined as the minimizer
$\tarstar \defn \arg \min_{\tarplain \in \TarClass} \EE \big[\loss(
  \tarplain(\Covariate), \Response) \big]$.  A portion of our theory
makes no assumptions about the particular form of $\fstar$, whereas a
more refined result applies to a class of GLM-based loss functions for
which this minimizer takes the form $\fstar(\covariate) = \psi \big(
\Exs[\Response \mid \Covariate = \covariate] \big)$ for some function
$\psi$.


\subsubsection{Conditions on loss functions}

We begin by specifying conditions on the loss functions that
underlie our analysis.  The first conditions are relatively standard
in empirical risk minimization
(Lipschitz and convexity), whereas the second condition (loss
compatibility) is more subtle, and highlights an important interplay
between the procedure used to estimate the ideal proxy $\auxstar$,
and the loss function $\Loss$ used to estimate $\fhat$.  We begin
with the former conditions.

\paragraph{Lipschitz and convexity conditions:}
We assume that the loss function satisfies certain Lipschitz and
convexity conditions.  First, we assume that it is
\emph{$\Lip$-Lipschitz} in both arguments, meaning that
\begin{subequations}
\begin{align}
\label{eqn:mreg-lip}
|\loss(\yhatzero, \yzero) - \loss(\yhatone, \yzero)| \leq \Lip
|\yhatzero - \yhatone|, \quad \text{and} \quad |\loss(\yhatzero,
\yzero) - \loss(\yhatzero, \yhatone)| \leq \Lip |\yhatzero - \yhatone|
\end{align}
for all quadruples $(\yhatzero, \yhatone, \yzero, \yone)$.  When $y$
and $\yhat$ are bounded, these conditions hold whenever $\loss$ has
first partial derivatives that are continuous.

Moreover, introducing the shorthand $\PP \loss_{\tarplain, \auxstar}
\defn \EE[ \loss(\tarplain(\Covariate), \auxstar(\Covariate,
  \Surrogate))]$, we also require that the function $\tarplain \mapsto
\PP \loss_{\tarplain, \auxstar}$ is \emph{$\scparam$-strongly convex}
at $\tarstar$, meaning that
\begin{align}
\label{EqnStrongConvexity}  
\PP(\loss_{\tarplain, \auxstar}) - \PP(\loss_{\tarstar, \auxstar})
\geq \frac{\scparam}{2} \| \tarplain - \tarstar \|_2^2, \qquad
\mbox{for all $\tarplain \in \TarClass$.}
\end{align}
\end{subequations}
This condition also holds for various loss functions used in
practice. \\

\medskip

\paragraph{Loss compatibility:}
Recall that the \ERIC~method allows for an arbitrary procedure for
estimating the auxiliary function $\auxstar$.  For a given loss
function $\loss$, we require that this auxiliary function $\auxstar$
is \emph{compatible} in the sense that
\begin{align}
\label{EqnAssumeTarstar}  
\underbrace{\arg \min_{\tarplain \in \TarClass} \EE \big[\loss(
    \tarplain(\Covariate), \Response) \big]}_{\equiv \tarstar} & =
\arg \min_{\tarplain \in \TarClass} \EE \big[\loss(
  \tarplain(\Covariate), \auxstar(\Covariate, \Surrogate)) \big].
\end{align}
For short, we say that the loss function $\Loss$ is
$\auxstar$-compatible when this condition holds.  As we discuss below,
this compatibility condition holds for generalized linear models as
well as a related class of Bregman divergences.  Thus, it includes the
least-squares loss, the logistic loss, and as well as related losses
for logistic, multinomial and exponential regression.

As we discuss in~\Cref{sec:loss-functions}, there are a broad class of
loss functions that are \mbox{$\auxstar$-compatible} with the choice
\begin{subequations}
\begin{align}
  \label{EqnIdeal}
\auxstar(\covariate, \surrogate) = \EE \big[ \Response \mid \Covariate
  = \covariate, \Surrogate = \surrogate \big].
\end{align}
This broader family includes (among others) all loss functions that
can be written in the form
\begin{align}
\label{EqnGenGLM}  
  \Loss(f(\covariate), \response) = -\phi(f(\covariate)) \response +
  \Phi(f(\covariate))
\end{align}
\end{subequations}
for some real-valued functions $\Phi$ and $\phi$.  Special cases of
the losses~\eqref{EqnGenGLM} include those that arise from generalized
linear models (GLMs), which lead to a special case with $\phi(s) = s$.
(See~\Cref{AppLossCompat} for details).  For this reason, we refer to
losses of the form~\eqref{EqnGenGLM} as \emph{GLM-type losses}.  An
important special case is logistic regression, which is based on a loss
function of the form~\eqref{EqnGenGLM} with $\phi(s) = s$ and $\Phi(t)
= \log(1 + e^t)$.  There are also various Bregman losses, including
the binary KL divergence, that can be written in the
form~\eqref{EqnGenGLM}; these and other examples are discussed
in~\Cref{sec:loss-functions}.

\subsubsection{Guarantee for general losses}
\label{SecGeneral}
We are now ready to state a guarantee that applies to any loss
function satisfying the previous conditions. We provide two
guarantees: one for general compatible losses and another for GLM-type
losses.  In the former setting (cf.~\Cref{thm:main}(a)), we study a
variant of the \ERIC~procedure in which we use the pseudo-responses
for the entire dataset---that is, with \mbox{$\Ytil_i \defn
  \auxhat(\Covariate_i, \Surrogate_i)$} for $i = 1, \ldots,
\numobs$. For the GLM-type losses, we study the \ERIC~procedure as
previously described (i.e., using pseudo-responses only for the
unlabeled set).

For this result, our high-order term takes the form
\begin{align}
  \label{EqnHighOrderTwo}
\HighOrder_{\numobs}(\pardelta) \defn \frac{12}{\sqrt{\numobs}} \sqrt{
  \tfrac{\Lip}{\scparam}} \; \sqrt{\log
  \big(\Logfun(\radcrit)/\pardelta \big)} \qquad \text{with} \qquad
\Logfun(\radcrit) \defn \log_2\big(\tfrac{4\fbound}{\radcrit}\big)
\end{align}
In addition to this higher order term, we also remind the reader of
the \emph{oracle accuracy} defined in equation~\eqref{EqnDefnRadcrit},
as well as the empirical norm $\|\cdot\|_\unumobs$ from~\cref{EqnNaiveMeasure}.

\begin{theorem}
\label{thm:main}
Under the conditions of~\Cref{thm:sqloss}, for any compatible,
\mbox{$\Lip$-Lipschitz} and \mbox{$\scparam$-strongly} convex loss,
consider the \ERIC~procedure using auxiliary estimate $\auxhat$.
\begin{enumerate}
\begin{subequations}
\item[(a)] For any compatible loss, it returns an estimate $\tarhat$
  such that
\begin{align}
\label{eqn:slow-rate}
\| \tarhat - \tarstar \|_2 \leq \Big(\tfrac{2 \Lip}{\scparam} + 1
\Big) \radcrit + \sqrt{\tfrac{8 \Lip}{\scparam} \; \PP_{\numobs}
  |\auxhat - \auxstar|} + \HighOrder_{\numobs}(\pardelta) 
\end{align}
with probability at least $1 - \pardelta$.

\item[(b)] For a GLM-type loss, it returns an estimate $\tarhat$ such that
\begin{align}
\label{eqn:faster-rate}
\|\tarhat - \tarstar \|_2 \leq \Big( \tfrac{2\Lip}{\scparam} + 1
\Big) \radcrit + \tfrac{2}{\scparam} \: \GLMbias + \Big(1 +
\sqrt{\tfrac{\Lip}{\scparam}}\Big) \:
\HighOrder_\numobs(\tfrac{\pardelta}{2})
\end{align}
\end{subequations}
with probability at least $1 - \pardelta$.  Here we define the
functions \mbox{$\htil(x) \defn \EE[\auxhat(\Covariate, \Surrogate)
    \mid \Covariate = x]$} and \mbox{$\hstar(x) \defn
  \EE[\auxstar(\Covariate, \Surrogate) \mid \Covariate = x]$.}
\end{enumerate}
\end{theorem}
\noindent See Appendices~\ref{sec:proof-main} and~\ref{AppGLM} for the proof of this theorem. \\

\medskip

\noindent We remark that~\Cref{thm:main}(b) holds for the broader
family~\eqref{EqnGenGLM} of GLM-type losses, as long as the function
$\phi$ satisfies some regularity conditions.  This family includes the
binary KL-divergence; see~\Cref{AppExtended} for further details. \\

The terms appearing in the bound~\eqref{eqn:faster-rate} have
analogous interpretations to those appearing in~\Cref{thm:sqloss}.  To
recapitulate briefly, the first term involving $\radcrit$ corresponds
a form of oracle risk, achievable by a procedure given access to
$\numobs$ labeled samples.  The third term is again a higher-order
term, arising from probabilistic fluctuations, whereas the second term
measures the pseudo-response defect via the difference $\GLMbias$.
Note that when the population minimizer takes the form
$\fstar(\covariate) = \Exs[\Response \mid \Covariate =
  \covariate]$---as is the case for the squared loss---then we have
the equivalence $\htil - \hstar = \ftil - \fstar$, consistent
with~\Cref{thm:sqloss}.  As for the bound~\eqref{eqn:slow-rate}, it is
a weaker guarantee, since the difference $\GLMbias$ has been replaced
by $\sqrt{\PP_{\numobs} |\auxhat - \auxstar|}$.  The proof
of~\Cref{thm:main}(b) exploits structure specific to GLM-type losses
so as to obtain the sharper guarantee~\eqref{eqn:faster-rate}.


\subsection{Behavior for binary classification}
\label{SecBinary}

The results in~\Cref{thm:main} validates our algorithm for a broad
class of problems and loss functions, of which one of particular
interest is \emph{binary classification}.  In this context, many
pipelines for large-scale machine learning are based on minimizing the
cross-entropy loss.  Equivalently, they are minimizing the binary
KL-divergence, for which the guarantees of of~\Cref{thm:main}(b) also
hold (see~\Cref{AppExtended} for details).  In this section, we
first show how our theory provides guidance on how the pseudo-labels
should be constructed (\Cref{sec:labeling-comparison}), and how the
\ERIC~procedure can extract useful information when noisy labels are
given as helper covariates (\Cref{sec:noisy-labeling}).


\subsubsection{Is it better to use hard or soft labels?}
\label{sec:labeling-comparison}

For classification with binary labels $\Response \in \{0, 1\}$, many
procedures for generating auxiliary predictors will return an estimate
of the conditional probability $\Prob[\Response = 1 \mid \covariate,
  \surrogate] = \Exs[\Response \mid \covariate, \surrogate]$, so that
$\auxfun(\covariate, \surrogate)$ be a real-valued scalar.  However,
many ``off-the-shelf'' machine learning algorithms for classification
are not equipped to handle continuous probability values.
Consequently, in order to make seamless use of existing pipelines, an
intermediate step of ``labeling'' the response $\auxhat(\covariate,
\surrogate)$---meaning using it to generate a binary label $\Ytil \in
\{0,1 \}$---is required. \\

\noindent Consider the following two approaches for generating such
binary labels:
\begin{cdesc}
\item[Hard labels:] Use $\auxhat$ to approximate the Bayes-optimal
  classifier---that is, generating
  \begin{align}
    \label{EqnHardLabel}
    \ResponseTil = \mathbf{1} \big[\auxhat(\covariate, \surrogate)
      \geq 1/2 \big] \defn \begin{cases} 1 & \mbox{if
        $\auxhat(\covariate, \surrogate) \geq 1/2$} \\ 0 &
      \mbox{otherwise.}
    \end{cases}
  \end{align}
\item[Stochastic soft labels:] Draw $\ResponseTil$ according to the
  Bernoulli distribution \mbox{$\ResponseTil \sim
    \text{Ber}(\auxhat(\covariate, \surrogate))$.}
\end{cdesc}
Hard labeling is natural approach for a practitioner.  However, as our
theory shows---and we illustrate here with a simple ensemble---the
hard labeling approach introduces bias into the overall procedure, and
hence mis-calibration in the final output $\fhat$ of the
\ERIC~procedure.

To illustrate this phenomenon, consider triples $(\Covariate,
\Surrogate, \Response)$ produced according the following procedure.
Given a covariate vector $\Covariate \in \RR^d$, we generate
\begin{subequations}
\begin{align}
\label{EqnHardSoftModel}
\Surrogate \sim \text{Ber}(\Wprob(\Covariate)), \qquad Z \sim
\text{Ber}(\Zprob(\Covariate)), \quad \text{and} \quad \Response = Z
\cdot \Surrogate.
\end{align}
Here the Bernoulli variable $Z$ is unobserved, and we have the
conditional independence relation $\Surrogate \indep Z \mid
\Covariate$.  If we use the binary KL loss, we can compute that the
population minimizer $\tarstar(\covariate) = \Exs[\Response \mid
  \covariate]$ and ideal proxy $\auxstar(\covariate, \surrogate) =
\Exs[\Response \mid \covariate, \surrogate]$ take the form
$\tarstar(\covariate) = \Wprob(\covariate) \cdot \Zprob(\covariate)$
and $\auxstar(\covariate, \surrogate) = \surrogate \cdot
\Zprob(\covariate)$.

In order to compare the behavior of the hard labeling and soft
labeling procedures, consider the following thought experiment:
suppose that we had access to the ideal proxy $\auxstar$.  We could
then generate either the hard labels $\HardY$ or the soft labels
$\SoftY$, as described above, using $\auxstar$.  The
bound~\eqref{eqn:faster-rate} from~\Cref{thm:main}(b) applies in
either case, with the bound for $\HardY$ and $\SoftY$ differing only
in the term\footnote{In the set-up given here, we have $\htil - \hstar
= \ftil - \fstar$.}  $\|\ftil - \fstar\|_2$, where
\begin{align*}
\fhard(\covariate) \defn \EE[\HardY \mid \Covariate = \covariate] =
\PP(\auxstar(\Covariate, \Surrogate) \geq \tfrac{1}{2} \mid \Covariate
= \covariate) = \Wprob(\covariate) \cdot \mathbf{1}(\Zprob(\covariate)
\geq \tfrac{1}{2})
\end{align*}
and
\begin{align*}
\fsoft(\covariate) \defn \EE[\SoftY \mid \Covariate =
  \covariate] = \EE[\EE[\SoftY \mid \Covariate, \Surrogate] \mid
  \Covariate = \covariate] = \Wprob(\covariate) \cdot
\Zprob(\covariate).
\end{align*}
Consequently, we see that when using the ideal proxy $\auxstar$, the
soft labels yield a perfectly calibrated $\ftil$---that is, $\fsoft
\equiv \fstar$---whereas using the hard labels yields a mis-calibrated
$\ftil$, since
\begin{align}
\label{EqnHardSoftBias}  
  \fhard(\covariate) - \fstar(\covariate) & = \Wprob(\covariate) \Big
  \{ \mathbf{1}(\Zprob(\covariate) \geq \tfrac{1}{2}) -
  \Zprob(\covariate) \Big \}.
\end{align}
This mis-calibration will affect the accuracy of the \ERIC~procedure
when using the hard labels.

In order to illustrate this phenomenon, we simulated from the
model~\eqref{EqnHardSoftModel} with covariate vectors $\Covariate \sim
\text{Unif}([0, 1]^5)$, and the choices
\begin{align}
  \label{EqnHardSoftChoice}
\Zprob(\covariate) = \sigmoid(\nu \langle \paramstar, \covariate
\rangle) \; \in \; [0,1], \quad \text{and} \quad \Wprob(\covariate) =
1 - 1.8|\sigmoid( \nu \langle \paramstar, \covariate \rangle) -
\tfrac{1}{2}| \: \, \in \: \, [0,1],
\end{align}
\end{subequations}
where $\sigmoid(t) = \tfrac{e^t}{1+e^{t}}$ is the sigmoid function.
The rationale for these choices of the functions $\Zprob$ and $\Wprob$
is to provide a varying amount of mis-calibration for the hard labels,
depending on the value of the parameter $\nu \geq 0$.  For values of
$\nu \approx 0$, we have $\Zprob(\covariate) \approx 1/2$, so that the
difference $\mathbf{1}(\Zprob(\covariate) \geq \tfrac{1}{2}) -
\Zprob(\covariate)$ is quite large.  At the same time, from the
definition of $\Wprob$, we see that $\Wprob(\covariate) \approx 1$
when $\nu \approx 0$.  By inspection of the mis-calibration
equation~\eqref{EqnHardSoftBias}, we see that these two properties in
conjunction mean that the mis-calibration is large when $\nu \approx
0$.  On the other hand, similar reasoning shows that when $\nu$ is
relatively large, then $\Wprob(\covariate) \approx 0$, so that the
mis-calibration is relatively small.

In order to verify these predictions, we simulated from this ensemble
with $\numobs = 1000$ samples, partitioned as $|\LabelSet| = \lnumobs
= 200$ and $|\UnlabelSet| = \unumobs = 800$. Both the training of
$\auxhat$ and $\tarhat$ were performed using random forest classifiers
with the cross-entropy loss.  As shown in~\Cref{FigNumSimClass}(a),
the behavior confirms what our theory predicts: as $\nu$ shrinks, the
error $\|\fhat - \fstar\|_2$ of the \ERIC~procedures when using hard
versus soft labels grows considerably.

\begin{figure}[h]
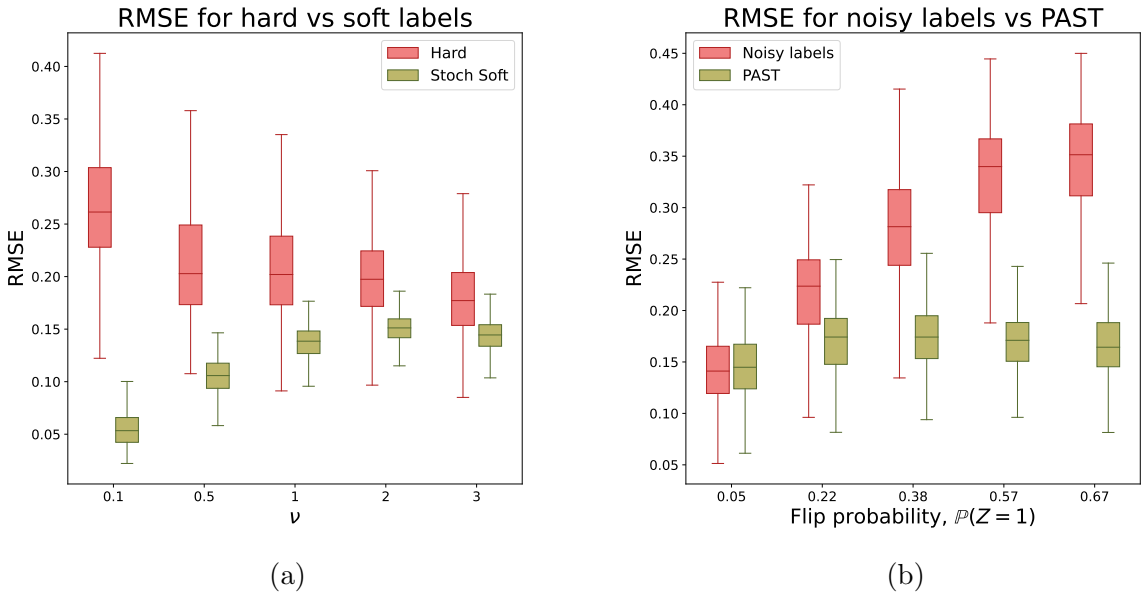

\begin{center}
\begin{tabular}{cc}
  \widgraph{0.5\textwidth}{\figdir/RMSE_boxplot} &
  \widgraph{0.5\textwidth}{\figdir/RMSE_flip_boxplot} \\
  (a) & (b)
\end{tabular} 
\caption{Verification of qualitative predictions for the
  \ERIC~procedure for binary classification.  settings. (a) Effect of
  using hard versus soft labels for the
  ensemble~\eqref{EqnHardSoftModel} with the
  choices~\eqref{EqnHardSoftChoice}.  Plots of the error $\|\ftil -
  \fstar\|_2$ for the \ERIC~procedure based on hard labels (red)
  versus soft labels (green).  Consistent with theory, use of
  stochastic soft labels results in a better calibrated estimate, with
  the greatest difference for $\nu \approx 0$ and shrinking
  differences as $\nu$ increases.  (b) Behavior of the \ERIC~procedure
  and direct classification approaches for the noisy label
  ensemble~\eqref{EqnNoisyLabelModel} with the
  choices~\eqref{EqnNoisyLabelChoices}.  Plots of the error $\|\ftil -
  \fstar\|_2$ for the \ERIC~procedure using the noisy labels as
  surrogates (green) versus the direct classification approach (red).
  Both procedures perform well when $\nu \approx 0$ so that $\Prob[Z =
    1] \approx 0$, but the direct classification degrades as $\nu$
  grows, while the performance of the \ERIC~procedure remains
  roughly constant.}
\label{FigNumSimClass}
\end{center}
\end{figure}

It should be emphasized that our theory---in providing guarantees on
the difference $\| \fhat - \fstar \|_2$---is focused on calibrated
classification.  We remark that if calibration is not the desired
objective, it is possible that using hard labels could lead to a
mis-calibrated classifier whose accuracy is superior than a calibrated
approach.  This phenomenon is observed to a small extent in some
real-world data described in the following section.


\subsubsection{Incorporating noisy labels}
\label{sec:noisy-labeling}

As discussed in the introduction, in many real-world settings, while
we may lack to high-quality labels $\Response \in \{0, 1\}$, we do
have access to low-quality or synthetic labels $\Surrogate \in \{0,
1\}$ of some kind. Various approaches to this problem have been
proposed, including methods based on weak
supervision~\cite{ratner2016data, ratner2017snorkel,
  robinson2020strength}, or methods that try to build in robustness to
noisy labels~\cite{natarajan2013learning, song2022learning}. All of
these methods involve considerable overhead, so in practice, when
the labels are not overly noisy, a standard approach is to treat them
as true responses for learning some classifier.  In this section, we
show that the \ERIC~procedure automatically makes use of noisy labels
in a seamless way, handling both the extremes of high accuracy and high
noise in an automated manner.  In particular, by construction, the
\ERIC~procedure can extract whatever statistical information lies in
the labels $\Surrogate$, even if they are highly noisy (for example,
always flipped relative to $\Response$). \\

As a concrete instance of a ``noisy label'' problem, consider triples
$(\Covariate, \Surrogate, \Response)$ that are produced with
$\Covariate \in \RR^d$, and
\begin{subequations}
\begin{align}
\label{EqnNoisyLabelModel}  
\Response \sim \text{Ber}(\Yprob(\Covariate)), \qquad Z \sim
\text{Ber}(\Zprob(\Covariate)), \quad \text{and} \quad \Surrogate =
\Response \oplus Z,
\end{align}
where $\oplus$ denotes the XOR operation (or addition mod 2).  As
before, the random variable $Z$ is unobserved, and satisfies the
conditional independence relation $\Surrogate \indep Z \mid
\Covariate$.  Note that that $Z \in \{0,1\}$ is an indicator for
whether or not the label $\Response$ is flipped in producing the noisy
label $\Surrogate$.

With these choices, we have $\fstar(\covariate) = \Exs[\Response \mid
  \Covariate = \covariate] = \Yprob(\covariate)$, which is properly
targeted by the \ERIC~procedure with $\Surrogate$ as the helper
covariate.  Suppose, instead, that we treated the helper covariates
$\Surrogate$ as true labels; such a procedure would be targeting the
function $\EE[\Surrogate \mid \Covariate = \covariate] =
\Yprob(\covariate)(1 - \Zprob(\covariate)) + (1 - \Yprob(\covariate))
\Zprob(\covariate)$.  We can compute that the resulting bias takes the
form
\begin{align}
\EE[\Response \mid \Covariate = \covariate] - \EE[\Surrogate \mid
  \Covariate = \covariate] = 2 \Zprob(\covariate) \big \{ \tfrac{1}{2}
- \Yprob(\covariate) \big \}.
\end{align}
This bias term is small as long as the flip probability
$\Zprob(\covariate)$ is small, but it can become large whenever
$\Yprob(\covariate)$ differs significantly from $1/2$, and
$\Zprob(\covariate)$ is not too small. \\

In order to exhibit the effect of this bias in practice, we drew
samples from the model~\eqref{EqnNoisyLabelModel} with $\Covariate
\sim \text{Unif}[0, 1]^5$, and the choices
\begin{align}
  \label{EqnNoisyLabelChoices}
\Yprob(\covariate) = \sigmoid( \langle \paramstar, \covariate \rangle)
\; \in \; [0,1], \quad \text{and} \quad \Zprob(\covariate) =
1.8|\sigmoid( \nu \langle \paramstar, \covariate \rangle) -
\tfrac{1}{2}| \: \, \in \: \, [0,1],
\end{align}
\end{subequations}
where $\sigmoid(t) = e^t/(1 + e^t)$ is the sigmoid function.  Here the
scalar parameter $\nu \geq 0$ controls the family of flip
probabilities $\Zprob(\covariate)$.  When $\nu \approx 0$, we
typically have $\Zprob(\covariate) \approx 0$, so that the helpers
$\Surrogate$ are directly informative of $\Response$.  As $\nu$
increases, then the probabilities $\Zprob(\covariate)$ spread out over
the range $[0, 0.9]$, depending on the underlying covariate.  A
procedure that treats $\Surrogate$ as true labels will decay as the
average missing probability $\Prob[Z = 1] =
\Exs_\Covariate[\Zprob(\Covariate)]$ approaches $1/2$, whereas the
\ERIC~procedure--- since it computes pseudo-labels based on the pair
$(\Surrogate, \Covariate)$---has the ability to learn which pairs
$(\surrogate, \covariate)$ are informative of the true label.
Consequently, we expect that the performance of the direct
classification procedure should degrade as we increase $\nu
\rightarrow +\infty$, whereas that of the \ERIC~procedure should be
roughly invariant.

In order to assess this prediction, we simulated this ensemble with
$\numobs = 1000$ samples, partitioned as $|\LabelSet| = \lnumobs =
200$ and $|\UnlabelSet| = \unumobs = 800$. Both the training of
$\auxhat$ and $\tarhat$ were performed using random forest classifiers
with cross-entropy loss.  We conducted experiments for $\nu \in \{
0.1, 0.5, 1, 2, 3 \}$, but plot results in terms of the missing
probability $\Prob[Z = 1]$.  As shown in~\Cref{FigNumSimClass}(b),
consistent with our theory, the \ERIC~procedure has error that remains
roughly invariant to the choice of $\nu$, whereas the direct
classification approach exhibits significantly growing error.


\section{Empirical results}
\label{sec:empirical}

We now turn to some real-world applications of the \ERIC~procedure, in
particular tackling four different prediction problems, drawn from a
diverse range of applications, for which: (a) it is natural to expect
that covariates may be observed without associated responses; and (b)
there is a natural notion of a helper covariate.  In all cases, we
implement both the \ERIC~procedure (cf. Algorithm 1), and compare it
to the ``naive'' method that involves training only the labeled
dataset $\LabelSet$; moreover, we use empirical risk minimization over
the same function class $\Fclass$ for both the \ERIC~method and the
naive method.\footnote{To be clear, our naive method is truly naive,
in that it ignores the unlabeled data; one could imagine applying
other more sophisticated methods, but this is beyond the scope of the
current paper.}

In~\Cref{sec:BRFSS}, we study the problem of forecasting societal ills
(e.g., alcoholism, drug addiction etc.) within cities. The data is
gathered through survey sampling, leading to issues with coverage and
missing responses. In~\Cref{sec:MIMIC}, we seek to classify whether or
not an individual will suffer a cardiovascular emergency within 6
months following a heart attack. One major problem within the medical
machine learning community is a lack of access to quality labels;
acquiring responses often entails convincing and paying medical
professionals to label data. In~\Cref{sec:NLSY}, we study the problem
of predicting an individual's income based on features collected
during their high school years.  In longitudinal data of this type,
responses are frequently missing due to the effect of
drop-out. Finally, in~\Cref{sec:MIMIC-CXR}, we seek to diagnose a
given patient with pneumonia based on a chest X-ray, where
high-quality labels are also difficult or costly to acquire.  This
challenge has led many researchers to develop machine learning methods
that use natural language processing (NLP) models to construct
synthetic labels based on electronic health records.  Such synthetic
labels can also be incorporated within our procedure as helpers;
recall our discussion from~\Cref{sec:noisy-labeling} on the ``noisy
label'' setting.


\subsection{Forecasting societal ills in communities}
\label{sec:BRFSS}

Societal ills such as alcoholism or addiction are challenging.  While
they can be reduced by suitable policy interventions, the resources
required to implement any policy need to be counterbalanced against
other societally beneficial uses of the same resources. In this
context, tools for forecasting the rates of societal ills are
extremely valuable to policy-makers.  Here the forecasting is of a
longitudinal nature: we seek the predict the fraction $Y \in [0,1]$ of
people exhibiting a certain trait (e.g., alcoholism) at a future time.
The standard covariates $\Covariate$ are features available at predict
time, whereas the helper covariates $\Surrogate$ might be the same
features available at a future time (but not available at predict
time).

As a concrete instantiation of such a set-up, we acquired data from
the 2010 American Community Survey (ACS)~\cite{ACS2010}; it contains
various types of demographic data for $\numobs = 461$ cities.  We
considered three separate prediction tasks, each distinguished by the
type of societal ill to be predicted; the responses for each of the
three tasks is the fraction $Y \in [0,1]$ in 2020 of individuals who
exhibit traits of $\{$alcoholism, smoking, obesity$\}$.  We obtained
the response data from the Center for Disease Control's (CDC)
Behavioral Risk Factor Surveillance System (BRFSS)~\cite{BRFSS2021};
it is a large scale survey conducted with the goal of gauging the
overall state of health within the United States. In particular, we
used the dataset ``Selected Metropolitan/Micropolitan Area Risk
Trends'' (SMART), which provides data for around 120 different cities.
We chose as the helper covariate $\Surrogate$ the $5$-year American
Community Survey taken between 2017-2021~\cite{ACS2021}. The ACS data
is accessed via the IPUMS National Historical Geographic Information
System (NHGIS)~\cite{censusdata}. \\

\noindent Summarizing the set-up:
\begin{cdesc}
\item[Sample sizes:] We have $\lnumobs = 91$ labeled samples, and
  $\unumobs = 371$ unlabeled samples.
\item[Responses:] The scalar $\Response \in [0, 1]$ indicates the
  fraction of people within a city in the survey that exhibit traits
  of alcoholism, smoking, or obesity

\item[Standard covariates:] The covariate vector $\Covariate \in
  \RR^{112}$ consists of different measurements pulled from the 2010
  ACS survey, including demographic information (e.g., age, gender,
  race) and other economic indicators (e.g., median income, industry
  of employment).

\item[Helper covariates:] The helper covariate $\Surrogate \in
  \RR^{112}$ consists of the same features used to form $\Covariate$,
  except taken from the more recent 2017-2022 ACS survey.

\item[Function class $\Fclass$:] Random forest regression with
  hyperparameters selected via cross-validation.

\item[Pseudo-responses:] Generated by random forest regression with
  hyperparameters selected via cross-validation.
\end{cdesc}

Table~\ref{table:BRFSS} describes the results, in particular providing
the $R^2$-values obtained by the \ERIC~method and the naive method for
each of the $3$ different responses.  Note that the \ERIC~method leads
to substantial improvements in accuracy on the test set; the increases
are of the order $50\%$ for prediction of alcoholism and overweight
rates, and a smaller but still nontrivial increase for predicting
smoking rates.

\begin{table*}[t]
\caption{Forecasting societal ills empirical results ($R^2$)}
\label{table:BRFSS}
\begin{center}
\begin{tabular}{ l @{\qquad \qquad \qquad} c @{\qquad \qquad} c}
\hline \hline & Naive & \ERIC \\ \hline Alcoholism & $0.147$ &
$0.216$ \\ Smoking & $0.586$ & $0.625$ \\ Overweight & $0.250$ &
$0.322$ \\ \hline \hline
\end{tabular}
\end{center}
\end{table*}


\subsection{Cardiovascular risk after heart attacks}
\label{sec:MIMIC}

Heart attacks are caused by plague buildup in coronary arteries, which
then restricts the flow of blood and oxygen to the heart. In this
context, we studied the following binary classification problem: how
to predict whether or not a patient, upon having suffered a myocardial
infarction\footnote{I.e., a heart attack} (MI), will return to the
emergency room (ER) within 6 months for a cardiovascular emergency?
Patients who suffer from an MI are known to have an elevated risk of a
cardiovascular-related emergency over the following 6--12
months~\cite{10.1093/eurheartj/ehu505}; consequently, an accurate
classifier for our task can be used to identify those patients who are
most likely to have such an event post-MI.  In order to do so, we used
data taken from MIMIC-IV dataset .  This
dataset~\cite{johnson2023mimic}, available at
PhysioNet~\cite{PhysioNet}, consists of ``de-identified patient''
records collected from the emergency department and intensive care
unit at the Beth Israel Deaconess Medical Center in Boston, MA. The
dataset contains a number of different features, including as
patients' visits to the ER, their diagnoses and procedures performed,
as well as lab tests and various other medical information. In this
setting, each sample represents a visit to the ER in which an MI was
the primary diagnoses. Upon arriving at an ER where an MI is
diagnosed, usually through an electrocardiogram (ECG), the patient is
prescribed some collection of drugs for treatment, and to prepare them
for (typically) a coronary angioplasty in which a catheter is inserted
and used to open up the blocked artery. The MIMIC-IV dataset is
well-maintained, meaning that there are no missing labels. However,
this desirable property is \emph{not the norm} for medical datasets;
indeed, acquiring labels typically requires extensive data collection
and manual labeling from medical professionals.  So as to emulate this
real-world setting, we constructed our partitioned labeled-unlabeled
datasets by randomly selecting a subset of labels to keep.

\begin{figure}[h!]
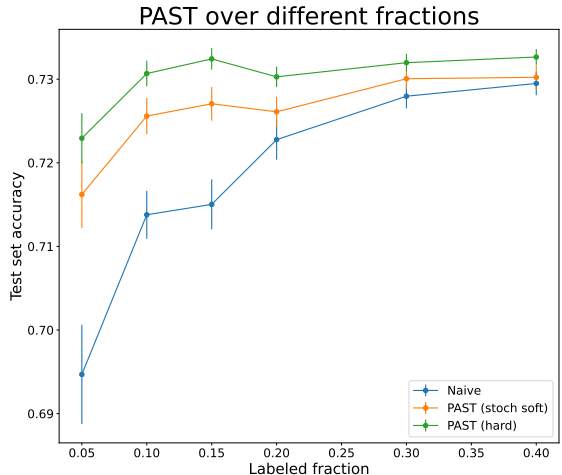

  \begin{center}
    \widgraph{0.55\textwidth}{\figdir/MIMIC_accuracy}
    \caption{Predicting subsequential cardiovascular risk after
      heart attacks.  Plots of the test set accuracy of binary classifiers
      based on training sets with varying fractions of labeled
      responses.  The naive procedure (blue) makes use only of the
      labeled data points, whereas the other two curves are
      instantiations of the \ERIC~procedure with either hard
      labels~\eqref{EqnHardLabel} (green) or stochastic soft labels
      (orange).  For a given fraction of labeled responses, error bars
      are computed by re-running the estimator on training sets
      constructed by randomly choosing the subset of observations to
      be labeled.}
\label{fig:mimic-results}
  \end{center}
\end{figure}

\noindent Summarizing the set-up:
\begin{cdesc}

\item[Sample sizes:] We have a total of $\numobs = 709$ samples, all
  labeled. We present our method by randomly selecting a subset of the
  data whose labels are kept, and the rest are dropped.

\item[Responses:] The binary label $\Response \in \{0, 1\}$ indicates
  whether the given patient suffered a cardiovascular emergency within
  6 months after having experienced a heart attack.

\item[Standard covariates:] The covariate vector $\Covariate \in
  \RR^{32}$ contains features for a given patient (that has suffered a
  heart attack), including age, past medical history, demographic
  information, etc.

\item[Helper covariates:] The surrogate covariate $\Surrogate \in
  \RR^{10}$ contains information about the patient's follow-up
  visit. It includes various indicators for several drugs typically
  prescribed to individuals undergoing cardiac emergencies (e.g.,
  aspirin or heparin).

\item[Function class $\Fclass$:] Random forest classification with
  hyperparameters selected via cross-validation.

\item[Pseudo-responses:] Generated by random forest classification
  with hyperparameters selected via cross-validation, and run for both
  stochastic soft and hard labeling.
\end{cdesc}

Since the entire dataset is labeled, we constructed random partitions
of labeled-unlabeled data points as follows.  For each fraction $p \in
\{0.05, 0.10, \ldots, 0.40 \}$, we chose uniformly at random a
fraction $p$ of the samples to be labeled, and removed labels from the
remaining $(1-p)$-fraction.  We then applied the \ERIC~procedure and
the naive procedure to this dataset, and repeated such a trial $T =
1000$ times.  In~\Cref{fig:mimic-results}, we plot the test set
classification accuracy versus the fraction $p$ for both methods.
Each point in this plot corresponds to the test accuracy averaged over
the $1000$ trials, along with corresponding error bars derived from
these trials.  In all cases, we use a random forest (both to generate
pseudo-responses, and to perform the final fit).  We applied the
\ERIC~procedure both with stochastic soft labels (orange curve) and
hard labels (green).  As shown in~\Cref{fig:mimic-results}, either
case leads to gains in classification accuracy for relatively small
values of $p$, with the gains decreasing as $p$ increases.
Interestingly, for this particular problem, the hard labeling
approach---which can lead to mis-calibrated classifiers, as discussed
previously---yields slightly better test accuracy.


\subsection{Educational longitudinal studies}
\label{sec:NLSY}

Next, we turn to a dataset based on a longitudinal survey, in which
measurements are collected from individuals over a period of distinct
time periods.  Missing data often arises due to dropout---that is,
beyond a certain point, individuals become unresponsive to survey
requests.  There is a rich literature on survival analysis, focusing
on estimation under this kind of censoring~\cite{klein2006survival}.
Here we consider an alternative approach, based on identifying a
suitable helper covariate, and then applying the \ERIC~procedure.

Concretely, the National Longitudinal Survey of Youth (NLSY) from the
1997 cohort~\cite{moore2000national} is longitudinal survey of almost
$9000$ men and women; it began in 1997 with follow-up surveys every two years (roughly). The survey consists of questions
about health, family status, education, employment and income, as well
as a variety of other records. We consider the problem of forecasting
an individual's income $\Response \in \real$ in the year 2004; in
order to do so, we make use of a covariate vector $\Covariate \in
\real^{33}$ of various features (e.g., time in high school, GPA,
standardized test scores, family status etc.)  As a helper covariate
$\Surrogate \in \real$, we use their income in the year 2002. \\

\noindent Summarizing the set-up:
\begin{cdesc}

\item[Sample sizes:] We have $\lnumobs = 497$ labeled observations and
  $\unumobs = 2023$ unlabeled observations.

\item[Responses:] The response $\Response \in [0, 100000]$ is the
  income of a given individual in 2004.

\item[Standard covariates:] The covariate vector $\Covariate \in
  \RR^{33}$ consists of various features measured in 1997 (e.g.,
  family information, GPA, standardized test scores, criminal
  activities etc.)

\item[Helper covariates:] For our helper $\Surrogate \in \RR$, we use
  the individual's income in 2002.

\item[Function class $\Fclass$:] Random forest regression with
  hyperparameters selected via cross-validation.

\item[Pseudo-responses:] Generated by random forest regression with
  hyperparameters selected via cross-validation.
\end{cdesc}

For this data, we find that the naive approach yields a model with an
$R^2$ of $4.6\%$, whereas the \ERIC~procedure yields an $R^2$ of
$6.5\%$, which is a non-trivial improvement.  (To clarify, low
$R^2$-values of the order given here are common in the social science
literature~\cite{ozili2023acceptable}.)


\subsection{Identifying pneumonia from chest X-rays}
\label{sec:MIMIC-CXR}

\begin{figure}[t!]
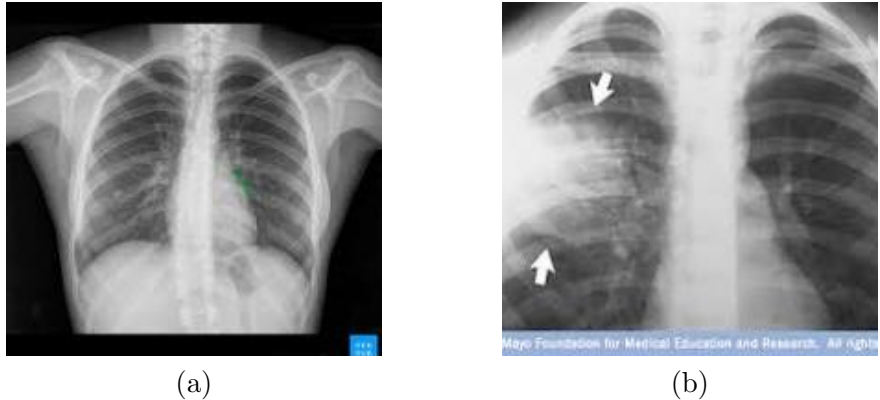

  \begin{center}
    \begin{tabular}{ccc}
      \widgraph{0.32\textwidth}{\figdir/healthy-CXR} & \hspace*{0.3in} &
      \widgraph{0.32\textwidth}{\figdir/pneumonia-CXR} \\
      (a) & & (b)
    \end{tabular}
    \caption{Some examples of chest X-rays. (a) A healthy chest
      X-ray. (b) A chest X-ray of a patient with pneumonia. The arrows
      point to the white spot corresponding to fluid-filled air sacs
      in the lungs, indicative of pneumonia.}
    \label{fig:CXR-examples}
  \end{center}
\end{figure}
Chest X-rays (CXRs) are a widely-used diagnostic tools, used to
identify various lung or heart conditions.  There are approximately 70
million CXRs performed within the United States each
year~\cite{iyeke2022reducing}, so automated procedures for
identifying occurrences of a given health condition based on a CXR could
dramatically improve the overall workflow and efficiency of hospitals.
However, high-quality labeled data is difficult to obtain, since it
requires the time and effort of a radiologist.  In recent years,
researchers have turned towards synthetic labels derived from natural
language processing (NLP) systems~\cite{irvin2019chexpert,
  peng2017negbio}, but these labels are very ``noisy'' relative to the
ground truth.  At the same time, the noisy labels provide a natural
choice of helper covariate for the \ERIC~procedure.

Concretely, we studied the binary classification problem of
predicting, on the basis of a CXR, whether or not the patient has
pneumonia.  It is a lung infection that causes air sacs to fill with
fluid, inhibiting breathing; see~\Cref{fig:CXR-examples} for
comparison of a healthy CXR to one indicating pneumonia.  In order to
do so, we used data taken from the MIMIC-CXR
database~\cite{johnson2019mimiccxr}; it consists of a collection of
CXRs that are each paired with a radiologist report.
As our helper covariate, we made use of synthetic labels from the NegBio
procedure; these are noisy labels, in the sense they had roughly
$8\%$ disagreement with the radiologist's labels.

\noindent Summarizing the set-up:
\begin{cdesc}
\item[Sample sizes:] We have $\lnumobs = 510$ labeled observations,
  and $\unumobs = 10000$ unlabeled observations.
\item[Responses:] $\Response \in \{0, 1\}$ is an indicator for
  whether the patient in the given X-ray has pneumonia or not.
\item[Standard covariates:] $\Covariate \in \RR^{1024}$ are the features
  produced by the foundation model for chest X-rays trained in the
  paper~\cite{cohen2022torchxrayvision}.
\item[Helper covariates:] The helper covariate $\Surrogate \in \{0,
  1\}$ is an indicator derived from the NegBio labels.
\item[Function class $\Fclass$:] $3$-layer neural network fit using
  SGD with the binary entropy loss.
\item[Pseudo-responses:] Generated by logistic regression with
  $\ell_1$-regularization using hard labels.
\end{cdesc}

\begin{figure}[h!]
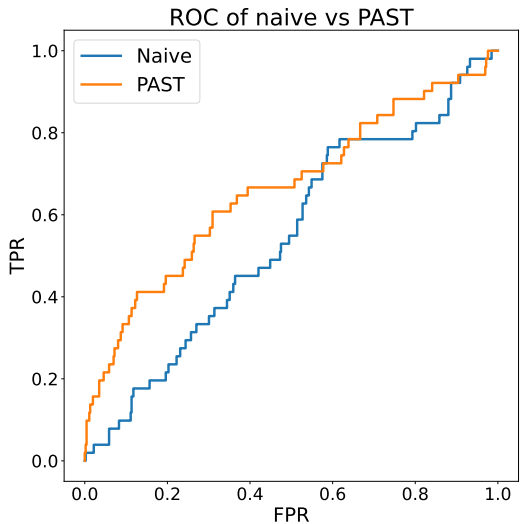

  \begin{center}
    \widgraph{0.5\textwidth}{\figdir/CXR_ROC_plots}
    \caption{ROC curves for the naive (blue) and \ERIC~(orange)
      approaches for identifying pneumonia from a chest X-ray.  True
      positive rate (or power) on the vertical axis versus the false
      positive rate (or Type I error) on the horizontal axis.
      Curves are obtained by varying the threshold used to make
      decisions for a given estimate $\fhat$.}
    \label{fig:mimic-CXR-results}
  \end{center}
\end{figure}

For a given method based on an estimate $\fhat$, we have a family of
binary classifiers, where we declare $\Yhat = 1$ if and only if
$\fhat(\Covariate) \geq \tau$ for some choice of threshold $\tau$.  As
the threshold $\tau$ is varied, we obtain a curve of the true positive
rate (or power) of the decision rule versus its false positive rate
(or Type I error).  These curves, known as ROC curves, are plotted in
\Cref{fig:mimic-CXR-results} for both the naive and \ERIC~methods.
Note that the \ERIC~procedure is superior to the naive procedure for
almost all choices of thresholds.  We can summarize the improvement by
computing the area underneath the ROC curve, a quantity known as the
AUC-ROC.  The naive procedure of only using the labeled data results
in an AUC-ROC equal to $0.55$, whereas using \ERIC~procedure yields an
AUC-ROC equal to $0.66$, a substantial improvement.


\section{Discussion}
\label{sec:conclusion}

In this paper, we have formalized a class of prediction problems with
missing responses, but with the availability of helper covariates.  This
framework includes a broad class of problems, among them forecasting
in time series with future information used as helper covariates;
incorporation of noisy or mis-calibrated predictions from pre-trained
machine learning models; and methods for transfer learning or
distribution shift.

We proposed a simple three-stage meta-procedure, known as the
\ERIC~method.  The first step is to train an auxiliary model $\auxhat$
to predict the response from the helper; second, use it to construct
pseudo-responses; and third, fit the final predictive model $\tarhat$
on the full augmented data set.  We obtain specific instantiations of
this meta-procedure by specifying how to train to auxiliary model, and
the loss function used to assess the quality of the final fit.  On the
theoretical side, we provided explicit and non-asymptotic guarantees
on the excess risk of the final fitted model $\tarhat$;
see~\Cref{thm:sqloss,thm:main}.  For a given loss function, our theory
identifies the notion of an ideal proxy $\auxstar$ that needs to be
well-approximated by the auxiliary estimate $\auxhat$.  Notably, for a
broad class of loss functions, our theory shows that final performance
depends on a $\Surrogate$-smoothed version of the difference $\auxhat
- \auxstar$; in particular, see equation~\eqref{EqnDefnFtil}. Our
theory also identifies an important notion of loss
consistency---between the proxy $\auxstar$ and the ideal predictor
$\fstar$----that needs to be satisfied.  We explored how the
\ERIC~method improves prediction for a variety of problems, ranging
from forecasting of alcoholism to diagnosis of pneumonia.

An advantage of the \ERIC~procedure is its simplicity, meaning that it
is straightforward to incorporate into existing ML pipelines for large-scale
prediction.  At the same time, it is easy to see ways in which its
performance might be improved (albeit with some sacrifice of
ease-of-use).  In particular, rather than simply fitting an auxiliary
model to produce pseudo-responses, we might model the full response
distributed conditioned on the standard-helper covariates, and make
use of the EM framework.  A more sophisticated approach of this type
could exploit, for example, any heterogeneity in the difficulty of
imputing responses as a function of the covariates.  Moreover, in our
current analysis, we have posited an independent form of data
generation, in which samples are generated i.i.d. with missing
responses chosen uniformly at random (also known as missing-completely
at-random).  We suspect that with some additional technical effort, it
could be relaxed to a milder missing-at-random
assumption.\footnote{More precisely, if we let $Z \in \{0,1\}$ be an
indicator of missingness, this assumption corresponds to the
conditional independence relation $\Response \indep Z \mid
(\Covariate, \Surrogate)$.}  For future work, it would be
interesting to consider this and other more general patterns
of missing responses.

Finally, our theory provides bounds that depend on both a form of
oracle risk---meaning the accuracy that could be achieved with a fully
labeled dataset---and a quality measure of the auxiliary fit
$\auxhat$.  We discussed conditions that ensure that the error in the
auxiliary fit is ``small enough'', meaning that the \ERIC~method
achieves the oracle risk up to constant factors.  However, as we
showed in~\Cref{SecEnsembleTwo}, it is possible to construct ensembles
for which the auxiliary error may be dominant, resulting in models
with poorer performance than the naive approach of using only labeled
data.  Thus, an important direction for future work is to develop
automated procedures for detecting such effects.  An obvious approach
is to guide the procedure via estimates of auxiliary error based on
hold-out and/or cross-validation; it would be interesting to develop
and study guided procedures in a more systematic way.

\subsection*{Acknowledgements}
EX was supported by an NSF Graduate Fellowship; in addition, MJW was
partially supported by ONR grant N00014-21-1-2842 and NSF DMS-2311072,
as well as by the Cecil H. Green Chair.



\AtNextBibliography{\small}
\printbibliography




\appendix

\section{Calculations for simple ensembles}
\label{AppEnsemble}

In this appendix, we collect together some simple calculations
that underlie the predictions given in~\Cref{SecEnsemble}.


\subsection{Verifying the prediction~\eqref{EqnGuaranteeOne}}
\label{AppEnsembleOne}

We begin by verifying the prediction~\eqref{EqnGuaranteeOne} for the
ensemble in~\Cref{SecEnsembleOne}.  For this ensemble, the ideal proxy
$\auxstar$ is given by $\auxstar(\covariate, \surrogate) =
\tarstar(\covariate) + \lambda \surrogate$, where
$\tarstar(\covariate) = \inprod{\betastar}{\Psi(\covariate)}$.  Thus,
we can obtain a consistent estimate of $\auxstar$ by performing linear
regression over functions of the form \mbox{$g_{\gamma}(\covariate,
  \surrogate) \defn \inprod{\gamma_1}{\Psi(\covariate)} + \gamma_2
  \surrogate$,} where $\gamma \defn (\gamma_1, \gamma_2) \in
\real^{\usedima + 1}$.

Letting $\gamhat$ denote the vector estimate, it defines the function
estimate $\auxhat \defn g_{\gamhat}$, and we have $\ftil(\covariate)
\defn \Exs \big[g_{\gamhat}(\Covariate, \Surrogate) \mid \Covariate =
  \covariate \big] = \inprod{\gamhat_1}{\Psi(\covariate)}$, using the
fact that $\Exs[\Surrogate \mid \Covariate = \covariate] =
0$. Consequently, recalling that $\fstar(\covariate) =
\inprod{\betastar}{\Psi(\covariate)}$, we have
\begin{align*}
\|\ftil - \tarstar\|_2^2 & = \Exs_{\Covariate} \inprod{\gammahat_1 -
  \betastar}{\Psi(\Covariate)} \leq (\gammahat_1 - \betastar) \CovMat
(\gammahat_1 - \betastar).
\end{align*}
where $\CovMat \defn \Exs[\Psi(\Covariate) \Psi^T(\Covariate)]$.
Finally, from equation~\eqref{EqnHelperVariance}, the procedure
leading to $\auxhat$ (and hence $\ftil$) is a $\usedima$-dimensional
linear regression estimate based on $\lnumobs$ samples, each
contaminated by noise with variance $(1-\lambda)^2 \sigma^2$.
Consequently, standard results on linear regression ensure that
$\|\ftil - \tarstar\|_2 = \|\sqrt{\CovMat} (\gammahat_1 -
\betastar)\|_2 \precsim \sigma (1-\lambda) \, \cdot \,
\sqrt{\frac{\usedima}{\lnumobs}}$.  Applying~\Cref{thm:sqloss} with
this bound together with the oracle accuracy $\radcrit \asymp \sigma
\sqrt{\usedima/\numobs}$ yields the claimed
bound~\eqref{EqnGuaranteeOne}.


\subsection{Verifying the prediction~\eqref{EqnGuaranteeTwo}}
\label{AppEnsembleTwo}

In this case, we obtain a consistent estimate of $\auxstar$ by
performing linear regression with functions of the form
$g_\gamma(\covariate, \surrogate) =
\inprod{\gamma_1}{\Psi(\covariate)} + \inprod{\gamma_2}{u} +
\inprod{\gamma_3}{v}$, where $\surrogate = (u,v)$, and $\gamma =
(\gamma_1, \gamma_2, \gamma_3) \in \real^\usedima \times \real \times
\real^\usedimb$.  Given a vector estimate $\gammahat$, we the function
$\gtil \defn g_{\gammahat}$. As in the calculation
in~\Cref{AppEnsembleTwo}, we have $\ftil(\covariate) =
\inprod{\gammahat_1}{\Psi(\covariate)}$, since $\Exs[\Surrogate \mid
  \Covariate = \covariate] = 0$ by construction.  Similarly, we can
argue as before that $\|\ftil - \tarstar\|_\unumobs \lesssim
\|\gammahat_1 - \betastar\|_2$, where we recall that
$\tarstar(\covariate) = \inprod{\betastar}{\Psi(\covariate)}$.  In
this case, it is less straightforward to bound the error
$\|\gammahat_1 - \betastar\|_2$, since we are simultaneously
estimating another high-dimensional quantity (namely, the vector
$\alphastar \in \real^{\usedimb}$).  However, we can make use on
recent non-asymptotic bounds for instrumental variable methods (see
the paper~\cite{xia2024iv}) to assert that $\|\gammahat_1 -
\betastar\|_2 \lesssim \sigma\ (1 - \lambda) \Big(
\sqrt{\frac{\usedima}{\lnumobs}} + \frac{\usedima +
  \usedimb}{\lnumobs} \Big)$.  The oracle accuracy for estimation of
$\fstar$ scales as $\sigma \sqrt{\usedima/\numobs}$.  Applying the
guarantee from~\Cref{thm:sqloss} with these pieces yields the
claim~\eqref{EqnGuaranteeTwo}.


\section{Loss functions in~\Cref{thm:main}}
\label{sec:loss-functions}

In this appendix, we discuss various loss functions to
which~\Cref{thm:main} applies.

\subsection{Loss function compatibility}
\label{AppLossCompat}

In this section, we characterize a wide range of loss functions that
are compatible with the function \mbox{$\auxstar(\covariate,
  \surrogate) \defn \Exs[\Response \mid (\Covariate, \Surrogate) =
    (\covariate, \surrogate)]$,} which we refer to as the
\emph{standard choice}.  Suppose that for any random variable $Z$, the
function $\hstar$ that minimizes the functional \mbox{$h \mapsto
  \Exs\Loss(h(\Covariate), Z)$} takes the form
\begin{align}
\label{EqnDefnHstar}  
\hstar(\covariate) & = \psi \big( \Exs[Z \mid \Covariate = \covariate]
\big),
\end{align}
where $\psi: \real \rightarrow \real$ is some fixed function.  We
claim that any such loss function is compatible with the standard
$\auxstar$.  Indeed, by applying condition~\eqref{EqnDefnHstar} with
$Z = Y$ and recalling the definition~\eqref{EqnPopulationRisk} of
$\tarstar$, we have the equivalence $\tarstar(\covariate) = \psi \big(
\Exs[Y \mid \Covariate = \covariate] \big)$. Similarly, applying
condition~\eqref{EqnDefnHstar} with $Z = \auxstar(\Covariate,
\Surrogate)$ yields
\begin{align*}
\arg \min_{\tarplain \in \TarClass} \EE \big[\loss(
  \tarplain(\Covariate), \auxstar(\Covariate, \Surrogate)) \big] \; =
\; \psi \big(\Exs[\auxstar(\Covariate, \Surrogate) \mid \Covariate = \covariate] \big)
\; = \; \tarstar(\covariate),
\end{align*}
where the final equality follows by the tower property of conditional
expectation. \\

Thus, it suffices to study loss functions that satisfy
condition~\eqref{EqnDefnHstar} for some $\psi$.  This property holds
for many loss functions; we consider two broad classes here.

\paragraph{Generalized linear models:}

Maximum likelihood using a generalized linear model (with canonical
link) leads to a loss function of the form
\begin{align}
\loss(f(x), y) & = - f(x) y + \Phi(f(x)).
\end{align}
Here $\Phi: \real \rightarrow \real$ is a convex function defined by
the GLM under consideration; it is differentiable with a monotonic
derivative $\Phi'$.  Note that these GLM losses are a special case of
the family~\eqref{EqnGenGLM} with $\phi(s) = s$.

We claim that condition~\eqref{EqnDefnHstar} holds if we define $\psi
\defn (\Phi')^{-1}$ to be the inverse of the derivative $\Phi'$.
Indeed, for any random variable $Z$, we can write
\begin{align*}
\Exs \big[- h(X) Z + \Phi(h(X))] \; = \; \Exs_{\Covariate} \big[ -
  h(X) \Exs[Z \mid X] + \Phi(h(X))].
\end{align*}
Taking derivatives pointwise for each fixed $\covariate$, we find that
the optimal $\hstar$ must satisfy the condition $\Exs[Z \mid
  \Covariate = \covariate] = \Phi'(\hstar(x))$, or equivalently
$\hstar(\covariate) = \psi \big( \Exs[Z \mid \Covariate = \covariate]
\big)$ with $\psi$ chosen as the inverse of $\Phi'$.

Let us consider some standard examples:
\begin{enumerate}
\item[(i)] The function $\Phi(t) = t^2/2$ corresponds to the
  least-squares loss with $\Yspace = \real$.  We have $\Phi'(t) = t$
  and $\psi(s) = s$, so that $\tarstar(\covariate) = \Exs[\Response \mid
    \Covariate = \covariate]$, as in our earlier analysis of
  least-squares.
\item[(ii)] The function $\Phi(t) = \log(1 + e^t)$ corresponds to a
  logistic regression model with \mbox{$\Yspace = \{0,1\}$.}  We have
  $\Phi'(t) = \frac{e^t}{1 + e^t} \in (0,1)$ and $\psi(s) =
  \log(s/(1-s))$ for $s \in (0,1)$, so that
  \begin{align}
\label{EqnLogOdds}    
\tarstar(\covariate) = \log \frac{\Exs[\Response \mid \Covariate =
    \covariate]}{1 - \Exs[\Response \mid \Covariate = \covariate]}.
\end{align}
\item[(c)] The function $\Phi(t) = e^t$ corresponds to a Poisson
  regression model with $\Yspace = \{0, 1, 2, \ldots \}$.  We have
  $\Phi'(t) = e^t$, and hence $\psi(s) = \log(s)$ for $s > 0$, so that
  $\tarstar(x) = \log \Exs[Y \mid \Covariate = \covariate]$.
\end{enumerate}

\paragraph{Bregman losses:}   In addition, there is a broad
class of Bregman losses for which condition~\eqref{EqnDefnHstar} holds
with the identity function $\psi(t) = t$.  These functions are related
to GLM-type losses; in many cases, they correspond to certain kinds of
dual formulations.  Let us consider a few examples:
\begin{enumerate}
\item[(a)] For $f(\covariate), y \in [0,1]$, the \emph{binary
  Kullback-Leibler divergence} is given by
  \begin{align}
\label{EqnBinaryKL}    
    \loss(f(\covariate), y) \defn y \log \frac{y}{f(\covariate)} +
    (1-y) \log \frac{1-y}{1 - f(\covariate)}.
  \end{align}
This objective is dual to the logistic regression loss; the difference
arises depending on whether we set up $\fstar$ as the log-odds
ratio~\eqref{EqnLogOdds}, or as the conditional expectation
$\fstar(\covariate) = \Exs[\Response \mid \Covariate = \covariate]$.
\item[(b)] For $f(\covariate), y > 0$, the Itakura-Saito divergence
  $\Loss(f(\covariate), y) = \frac{y}{f(\covariate)} - \log
  \frac{y}{f(\covariate)} - 1$.
\end{enumerate}
For all of these Bregman divergences, the
condition~\eqref{EqnDefnHstar} holds with $\psi(s) = s$.

\subsection{An extended family of loss functions}
\label{AppExtended}
Recall the form~\eqref{EqnGenGLM} of the generalized GLM loss
function, involving the term $\phi \circ f(\covariate) =
\phi(f(\covariate))$. In order for the guarantee~\Cref{thm:main}(b) to
hold, a careful inspection of our proof reveals that we require only
$\| \phi \circ \tarhat - \phi \circ \tarstar \|_2 \leq \| \tarhat -
\tarstar \|_2$, which is milder than a pointwise Lipschitz condition
on $\phi$.

An important special case is the loss function $\loss(f(\covariate),
\response) = - \response \log(\tfrac{f(\covariate)}{1-f(\covariate)})
- \log(1 - f(\covariate))$.  Risk minimization using this function is
equivalent to using the binary KL-divergence~\eqref{EqnBinaryKL}, also
known as the \emph{cross-entropy loss} in the machine learning
literature. Indeed, many modern machine learning algorithms for
classification involving fitting decision trees or neural networks to
minimize this particular loss.  This loss function is of the
form~\eqref{EqnGenGLM} with $\phi(s) = \log \big( \tfrac{s}{1- s}
\big)$, which is Lipschitz on any interval of the form $[c, 1- c]$ for
$c \in (0, 1/2)$.


\section{Proof of~\Cref{thm:sqloss}}
\label{sec:proof-sqloss}

Our proof involves the (squared) empirical norm $\totempnorm{f}^2 =
\PempTot[f^2(\Covariate)] \defn \frac{1}{\numobs} \sum_{i=1}^\numobs
f^2(\Covariate_i)$ over the full dataset, as well as the empirical
expectation operators $\PempLab(g) \defn \frac{1}{\lnumobs}
\sum_{i=1}^\lnumobs g(\Covariate_i, \Surrogate_i)$ and $\PempUnlab(g)
\defn \frac{1}{\unumobs} \sum_{i=\lnumobs+1}^\numobs g(\Covariate_i,
\Surrogate_i)$ over the $\lnumobs$-sized and $\unumobs$-sized datasets
$\LabelSet$ and $\UnlabelSet$ of labeled and unlabeled samples,
respectively.

\subsection{Main argument}

We begin with a lemma that provides an upper bound on the squared
error $\|\tarhat - \tarstar\|_2^2$ consisting of three terms:
\begin{lemma}
\label{LemDecomposition2}
We have the upper bound $\|\tarhat - \tarstar \|_2^2 \leq \sum_{j=1}^4
\Term_j$, where
\begin{subequations}
\begin{align}
\label{EqnDefnTermOne2}    
\Term_1 & \defn \Big \{ \| \tarhat - \auxstar \|_2^2 -
\totempnorm{\tarhat - \auxstar }^2 \Big \} - \Big \{ \| \tarstar -
\auxstar \|_2^2 - \totempnorm{ \tarstar - \auxstar}^2 \Big \}, \\
\label{EqnDefnTermTwo2}
\Term_2 & \defn \frac{2\lnumobs}{\numobs} \cdot \PempLab\big\{
(\tarhat - \tarstar)(\Response - \auxstar) \big\}, \\
\label{EqnDefnTermThree2}
\Term_3 & \defn \frac{2\unumobs}{\numobs} \cdot \PempUnlab\big\{
(\tarhat - \tarstar)(\auxhat - \ftil + \tarstar - \auxstar) \big\},
\quad \mbox{and}  \\
\label{EqnDefnTermFour2}
\Term_4 & \defn \frac{2\unumobs}{\numobs} \cdot \EmpNormUnlab{\tarhat
  - \tarstar} \cdot \EmpNormUnlab{\ftil - \tarstar}.
\end{align}
\end{subequations}
\end{lemma}
\noindent
See~\Cref{sec:proof-LemDecomposition2} for a proof of this lemma.

\medskip

With this decomposition in hand, we need to obtain suitable bounds on
each of the four terms.  These bounds, proved in the appendices,
involve both the critical radius $\radcrit$ for $\TarClass$ over the
full dataset, as well as their analogues $\lradcrit$ and $\uradcrit$
defined over the $\lnumobs$-sized and $\unumobs$-sized datasets
$\LabelSet$ and $\UnlabelSet$, respectively.  Moreover, we make use of
the shorthand $\Logfun(t) \defn \log_2\big(\tfrac{4B}{t}\big)$.  Note
that by our assumptions on the data generation process, the
distribution of the pairs $(\Covariate_i, \Surrogate_i)$ remains the
same across $\LabelSet$ and $\UnlabelSet$.

\begin{lemma}[Bound on $\Term_1$]
\label{LemTermOne}
Conditional on $\|\tarhat - \tarstar \|_2 \geq \radcrit$, we have
\begin{align}
\label{EqnTermOneBound}  
\Term_1 \leq 2 \, \|\tarhat - \tarstar \|_2 \Big( \radcrit + 4
\sqrt{\frac{2\log(\Logfun(\radcrit)/\pardelta)}{\numobs}} + \frac{128
  \log(\Logfun(\radcrit)/\pardelta)}{\radcrit \numobs} \Big)
\end{align}
with probability at least $1 - \pardelta$.
\end{lemma} 
\noindent See~\Cref{AppProofLemTermOne} for the proof of this
lemma. \\

\begin{lemma}[Bound on $\Term_2$]
\label{LemTermTwo}
Conditional on $\|\tarhat - \tarstar \|_2 \geq \radcrit$, we have
\begin{align}
\label{EqnBoundTermTwo}  
\Term_2 & \leq \|\tarhat - \tarstar\|_2 \cdot \Big( 2\sigbound
\radcrit + 4 \sigbound \sqrt{\frac{2
    \log(\Logfun(\radcrit)/\pardelta)}{\numobs}} \Big) + \frac{32
  \sigbound \log(\Logfun(\radcrit)/\pardelta)}{\numobs},
\end{align}
with probability at least $1 - \pardelta$.
\end{lemma}
\noindent See~\Cref{AppProofLemTermTwo} for the proof of this lemma.

\begin{lemma}[Bound on $\Term_3$]
\label{LemTermThree2}
Conditional on $\|\tarhat - \tarstar \|_2 \geq \radcrit$, we have
\begin{align}
\begin{split}
\Term_3 \leq 2 \|\tarhat - \tarstar \|_2 \Big( 2\radcrit + 8
\sqrt{\frac{2\log(\Logfun(\radcrit)/\pardelta)}{\numobs}} \Big) +
\frac{128\log(\Logfun(\radcrit)/\pardelta)}{\numobs},
\end{split}
\end{align}
with probability at least $1 - \pardelta$.
\end{lemma}
\noindent See~\Cref{sec:proof-LemTermThree2} for the proof of this
lemma. In particular, we can rewrite this guarantee as, conditional on $\|\tarhat - \tarstar \|_2 \geq \radcrit$, we have 
\begin{align*}
\Term_3 \leq 2 \|\tarhat - \tarstar \|_2 \Big( 2\radcrit + 8
\sqrt{\frac{2\log(\Logfun(\radcrit)/\pardelta)}{\numobs}}  +
\frac{64\log(\Logfun(\radcrit)/\pardelta)}{\numobs\radcrit} \Big)
\end{align*}

\begin{lemma}[Bound on $\Term_4$]
  \label{LemTermThree}
Conditional on $\|\tarhat - \tarstar\|_2 \geq \radcrit$, we have
\begin{align*}
\frac{2\unumobs}{\numobs} \cdot \|\tarhat - \tarstar\|_\unumobs  \leq 2\|\tarhat - \tarstar\|_2 +  4 \radcrit + 8 \sqrt{\frac{2\log(\Logfun(\radcrit)/\pardelta)}{\numobs}} + \frac{128 \log(\Logfun(\radcrit)/\pardelta)}{\numobs \radcrit},
\end{align*}
with probability at least $1 - \pardelta$.
\end{lemma}
\noindent See~\Cref{AppProofLemTermThree} for the proof. \\

\medskip

With these auxiliary results in hand, we can now complete the proof of
the main claim. Let us summarize in a compact way our conclusions thus far, making use
of the shorthand
\begin{align*}
\STerm_\numobs \defn 2\max\{\sigbound, 1\} \: \radcrit + 4\max\{ \sigbound, 2 \} \sqrt{\frac{2\log(\Logfun(\radcrit)/\pardelta)}{\numobs}} + 32 \max\{\sigbound, 8\} \: \frac{\log(3\Logfun(\radcrit)/\pardelta)}{\radcrit \numobs}.
\end{align*}
Using Lemmas~\ref{LemTermOne},~\ref{LemTermTwo},~\ref{LemTermThree},
and~\ref{LemTermThree2}, we have
\begin{align*}
\Term_1 & \leq \|\tarhat - \tarstar\|_2 \cdot \STerm_\numobs, \quad
\Term_2 \leq \| \tarhat - \tarstar \|_2 \cdot \STerm_\numobs, \quad
\Term_3 \leq 2 \|\fhat - \fstar \|_2 \cdot \STerm_\numobs \\ &\qquad
\text{and} \qquad \Term_4 \leq 2 \EmpNormUnlab{\ftil - \tarstar}
\Big\{ \|\tarhat - \tarstar \|_2 + \STerm_\numobs \Big\}.
\end{align*}
Therefore we have either $\|\tarhat - \tarstar\|_2 \leq \radcrit$ or
\begin{align*}
\| \tarhat - \tarstar \|_2^2 \leq 2 (2\STerm_\numobs +
\EmpNormUnlab{\ftil - \tarstar} ) \cdot \|\tarhat - \tarstar \|_2 +
2\STerm_\numobs \cdot \EmpNormUnlab{\ftil - \tarstar}
\end{align*}
Thus, conditional on $\|\tarhat - \tarstar\|_2 \geq \radcrit$,
rearranging the square and some basic algebra yields
\begin{align*}
\Big( \|\tarhat - \tarstar\|_2 - \big(2 \STerm_\numobs +
\EmpNormUnlab{\ftil - \tarstar} \big) \Big)^2 &\leq \big(2
\STerm_\numobs + \EmpNormUnlab{\ftil - \tarstar} \big)^2 +
2\STerm_\numobs \cdot \EmpNormUnlab{\ftil - \tarstar} \\ &\leq 2
\big(2 \STerm_\numobs + \EmpNormUnlab{\ftil - \tarstar} \big)^2
\end{align*}
with probability at least $1 - \tfrac{4\pardelta}{3}$. Consequently,
we have
\begin{align*}
\|\tarhat - \tarstar \|_2 \leq (1 + \sqrt{2}) \big(2\STerm_\numobs +
\EmpNormUnlab{\ftil - \tarstar} \big),
\end{align*}
and the claim follows from adding $\radcrit$ to the above expression.


\subsection{Proof of~\Cref{LemDecomposition2}}
\label{sec:proof-LemDecomposition2}

Recall the definition $\Term_1 \defn \big \{ \| \tarhat - \auxstar
\|_2^2 - \totempnorm{\tarhat - \auxstar }^2 \big \} - \big \{ \|
\tarstar - \auxstar \|_2^2 - \totempnorm{ \tarstar - \auxstar}^2 \big
\}$.  Our proof consists of two steps.  First, we show that
\begin{subequations}
\begin{align}
\label{EqnKeyDecomp}  
 \| \tarhat - \tarstar \|_2^2 & = \Term_1 +
 \underbrace{\totempnorm{\tarhat - \auxstar}^2 - \totempnorm{\tarstar
     - \auxstar}^2}_{\revdefn \Term_5}.
\end{align}
Next, we show that
\begin{align}
\label{EqnTermFiveUpper}  
\Term_5 & \leq \Term_2 + \Term_3 + \Term_4,
\end{align}
\end{subequations}
where the terms $\Term_2$, $\Term_3$ and $\Term_4$, were previously defined in
equations~\eqref{EqnDefnTermTwo2},~\eqref{EqnDefnTermThree2}, and~\eqref{EqnDefnTermFour2},
respectively.

\paragraph{Proof of the decomposition~\eqref{EqnKeyDecomp}:}
By definition, we have \mbox{$\auxstar(\covariate, \surrogate) \defn
  \EE[\Response \mid \covariate, \surrogate]$,} and hence \mbox{$\EE
  \big[ \auxstar(\Covariate, \Surrogate) \mid \covariate]
  \stackrel{(i)}{=} \EE[\Response \mid \covariate] \;
  \stackrel{(ii)}{=} \tarstar(\covariate)$.}  where step (i) follows
by iterated expectation; and step (ii) follows from the definition of
$\tarstar$.  As a consequence, we have
\begin{align*}
\inprod{\tarhat - \tarstar}{\tarstar - \auxstar}_2 & \defn \Exs
\Big[(\tarhat(\Covariate) - \tarstar(\Covariate)) \;
  (\tarstar(\Covariate) - \auxstar(\Covariate, \Surrogate) \Big] \;
= \; 0.
\end{align*}
Using this fact, we can compute
\begin{align*}
  \| \tarhat - \auxstar \|_2^2 = \| \tarhat - \tarstar + \tarstar -
  \auxstar \|_2^2 & = \| \tarhat - \tarstar \|_2^2 + 2 \inprod{\tarhat
    - \tarstar}{\tarstar - \auxstar}_2 + \|\tarstar - \auxstar \|_2^2
  \\
& = \| \tarhat - \tarstar \|_2^2 + \| \tarstar - \auxstar \|_2^2.
\end{align*}
Re-arranging this equality and adding/subtracting $\Term_5$ yields the
claim~\eqref{EqnKeyDecomp}.

\bigskip

\paragraph{Proof of the upper bound~\eqref{EqnTermFiveUpper}:}

By definition, the ERM procedure is minimizing the objective function
\mbox{$\tarplain \mapsto \totempnorm{\ResponseTil - \tarplain}^2 =
  \frac{1}{\numobs} \sum_{i=1}^\numobs \big(\ResponseTil_i -
  \tarplain(\Covariate_i) \big)^2$.}  By definition, the function
$\tarhat \in \TarClass$ is the constrained minimizer of this
objective, whereas $\tarstar$ is feasible.  Consequently, we have
$\|\ResponseTil - \tarhat\|_{\numobs}^2 \leq \| \ResponseTil -
\tarstar \|_{\numobs}^2$, whence
\begin{align*}
\| \tarstar - \tarhat + \ResponseTil - \tarstar \|_{\numobs}^2 \leq \|
\ResponseTil - \tarstar \|_{\numobs}^2.
\end{align*}
Expanding the square and re-arranging yields
\begin{subequations}
  \begin{align}
\label{EqnHackOne}    
\| \tarhat - \tarstar \|_{\numobs}^2 & \leq 2 \PempTot \Big\{(\tarhat
- \tarstar)(\ResponseTil - \tarstar) \Big\} = 2 \PempTot \Big\{
(\tarhat - \tarstar)(\auxstar - \tarstar) \Big\} + 2 \PempTot
\Big\{ (\tarhat - \tarstar)(\ResponseTil - \auxstar) \Big\}.
\end{align}
Second, we observe that
\begin{align}
\label{EqnHackTwo}  
\totempnorm{\tarhat - \auxstar}^2 & = \| \tarhat - \tarstar
\|_\numobs^2 + 2 \PP_\numobs \Big\{ (\tarhat - \tarstar)(\tarstar -
\auxstar) \Big\} + \| \tarstar - \auxstar \|_\numobs^2
\end{align}
\end{subequations}
Since $\Term_5 = \totempnorm{\tarhat - \auxstar}^2 -
\totempnorm{\tarstar - \auxstar}^2$, combining
inequality~\eqref{EqnHackOne} with equality~\eqref{EqnHackTwo} yields
\begin{align*}
\totempnorm{\tarhat - \auxstar}^2 - \totempnorm{\tarstar - \auxstar}^2 &\leq 2\PP_\numobs \Big\{ (\fhat - \fstar)(\ResponseTil - \auxstar) \Big\} \\
&= \frac{2\lnumobs}{\numobs} \cdot \PempLab \big\{(\tarhat -
\tarstar)(\Response - \auxstar) \big\} + \frac{2\unumobs}{\numobs} \cdot
\PempUnlab \big\{ (\tarhat - \tarstar)(\auxhat - \auxstar) \big\} \\ &=
\Term_2 + \frac{2\unumobs}{\numobs} \cdot \PempUnlab \big\{ (\tarhat
- \tarstar)(\auxhat - \auxstar) \big\}
\end{align*}
Then some basic algebra yields
\begin{align*}
\frac{2\unumobs}{\numobs} \: \PempUnlab \big\{ (\tarhat - \tarstar)
(\auxhat - \auxstar) \big\} &= \frac{2\unumobs}{\numobs} \: \PempUnlab
\big \{ (\tarhat - \tarstar)(\auxhat - \ftil + \tarstar - \auxstar)
\big\} + \frac{2\unumobs}{\numobs} \: \PempUnlab \big \{ (\tarhat -
\tarstar)(\ftil - \tarstar) \big \} \\ &= \Term_3 +
\frac{2\unumobs}{\numobs} \: \PempUnlab \big \{ (\tarhat -
\tarstar)(\ftil - \tarstar) \big \} \\ &\leq \Term_3 +
\frac{2\unumobs}{\numobs} \: \EmpNormUnlab{\tarhat - \tarstar} \cdot
\EmpNormUnlab{\ftil - \tarstar},
\end{align*}
as desired.


\subsection{Proof of~\Cref{LemTermOne}}
\label{AppProofLemTermOne}

Consider a function $h: \Xspace \times \Wspace \rightarrow \real$, and
associated samples $\{(\Covariate_i, \Surrogate_i)) \}_{i=1}^\numobs$.
Throughout this proof, we make use of the shorthand notation $\PempTot
h^2 \defn \frac{1}{\numobs} \sum_{i=1}^\numobs h^2(\Covariate_i,
\Surrogate_i) \equiv \totempnorm{h}^2$ and $\Prob h^2 \defn \Exs
\big[h^2(\Covariate, \Surrogate)\big] \equiv \|h\|_2^2$, along with
analogous notation for functions $f: \Xspace \rightarrow \real$.

Note that $\Term_1$ involves the function $\tarhat$, which is
data-dependent.  Consequently, in order to bound $\Term_1$, we need to
define a suitable empirical process, and bound its supremum.  So as to
do so, we first define the function
\begin{align*}
H_\tarplain(\covariate, \surrogate) \defn \big \{
\tarplain(\covariate) - \auxstar(\covariate, \surrogate) \big \}^2 -
\big \{ \tarstar(\covariate) - \auxstar(\covariate, \surrogate) \big
\}^2 \quad \mbox{for each $\tarplain \in \TarClass$.}
\end{align*}
Moreover, for each radius $\myrad > 0$, we define
\begin{align*}
  \supRVn{\myrad} \defn \sup_{H_f \in \hclass(\myrad)} \Big \{
  \PempTot H_f^2 - \Prob H_f^2 \Big \} \quad \mbox{where
    \mbox{$\hclass(\myrad) \defn \Big\{ H_\tarplain \, \mid \,
      \tarplain \in \TarClass \; \mbox{such that $\|\tarplain -
        \tarstar \|_2 \leq \myrad$.}  \Big \}$.}}
\end{align*}
Notice that $\supRVn{\myrad}$ is the supremum of a zero-mean empirical
process over $\hclass(\myrad)$, measuring the deviations between the
$L^2(\Prob)$-norm and its empirical counterpart $L^2(\PempTot)$.
Moreover, by construction, we have the upper bound
\begin{subequations}
\begin{align}
\label{EqnRandomRad}
\Term_1 & \leq \supRVn{\|\tarhat - \tarstar\|_2}.
\end{align}

With this set-up, the remainder of the proof consists of two steps.
First, we prove that for each fixed radius $\myrad \geq \radcrit$,
\begin{align}
\label{EqnFixedRadBound}
\supRVn{\myrad} \leq \radcrit \: \myrad + 4 
\sqrt{\frac{2\log(1/\pardelta)}{\numobs}} \: \myrad + \frac{64 \log(1/\pardelta)}{\numobs}.
\end{align}
\end{subequations}
This bound holds for a fixed radius $\myrad$, so that that we cannot
directly apply it to the random radius $\|\tarhat - \tarstar\|_2$ in
the inequality~\eqref{EqnRandomRad}.  However, we can
apply~\Cref{lems:peeling}, a general result on ``peeling'' proved
in~\Cref{AppLemPeeling}.  Using 
\begin{align*}
Q(\myrad, s) \defn \radcrit \myrad + 4 \sqrt{\frac{2s}{\numobs}} \: \myrad + \frac{64\log(1/\pardelta)}{\numobs}, \qquad \text{and} \qquad U = \|\tarhat - \tarstar \|_2, 
\end{align*}
we conclude that if
bound~\eqref{EqnFixedRadBound} holds, then have
\begin{align*}
\supRVn{\|\tarhat - \tarstar \|_2} & \leq 2 \radcrit \|\tarhat
- \tarstar \|_2 + 8 
\sqrt{\frac{2\log(\Logfun(\radcrit)/\pardelta)}{\numobs}} \cdot \|
\tarhat - \tarstar \|_2 + \frac{128
  \log(\Logfun(\radcrit)/\pardelta)}{\numobs}
\end{align*}
with probability at least $1 - \pardelta$. Combined with our original
bound~\eqref{EqnRandomRad}, this completes the proof of the lemma.\\


\noindent It remains to prove our outstanding claim.

\paragraph{Proof of the bound~\eqref{EqnFixedRadBound}:}
Fix some $\myrad \geq \radcrit$, and adopt the shorthand $\hclass
\equiv \hclass(\myrad)$.  Recall that $\VarFun{\hclass} \defn \sup_{h
  \in \hclass} \Var(h)$.  Applying~\Cref{lems:emp-conc} with $\tau =
1$ to $\hclass$ yields
\begin{align*}
\supRVn{\myrad} \leq 2 \, \EE[\supRVn{\myrad}] +
\sqrt{\VarFun{\hclass}} \sqrt{\frac{2\log(1/\pardelta)}{\numobs}} +
\frac{64\log(1/\pardelta)}{\numobs}
\end{align*}
with probability at least $1 - \pardelta$. In order to complete the
proof, it suffices to show that
\begin{align}
\label{eqn:emp-calcs}
\EE[\supRVn{\myrad}] \stackrel{(a)}{\leq}  \: \myrad \:
\radcrit, \qquad \text{and} \qquad \VarFun{\hclass}
\stackrel{(b)}{\leq} 16  \: \myrad^2,
\end{align}
valid for any $\myrad \geq \radcrit$. \\

\noindent \emph{Variance bound:} We first prove the variance upper
bound~\eqref{eqn:emp-calcs}(b).  We have
\begin{align*}
\VarFun{\hclass(\myrad)} = \sup_{\CompFun \in \hclass(\myrad)}
\VarFun{h} & \leq \sup_{\|\tarplain - \tarstar \|_2 \leq \myrad} \PP
\Big( (\tarplain - \auxstar)^2 -(\tarstar - \auxstar)^2 \Big)^2 \\
& = \sup_{\|\tarplain - \tarstar \|_2 \leq r} \PP \Big[ (\tarplain -
  \tarstar)^2 (\tarplain + \tarstar - 2 \auxstar)^2 \Big] \\
& \leq 16  \sup_{\|\tarplain - \tarstar \|_2 \leq \myrad}
\PP(\tarplain - \tarstar)^2 = 16 \myrad^2,
\end{align*}
as claimed. \\

\noindent \emph{Mean bound:} Next we prove the upper
bound~\eqref{eqn:emp-calcs}(a) on the mean.  We have
\begin{align*}
\EE[\supRVn{\myrad}] & = \EE\Big[ \sup_{\|\tarplain - \tarstar \|_2
    \leq \myrad} \Big| \|\tarplain - \auxstar \|_{\numobs}^2 - \|
  \tarstar - \auxstar \|_{\numobs}^2 - \big(\| \tarplain - \auxstar
  \|_2^2 - \| \tarstar - \auxstar \|_2^2 \big) \Big| \Big] \\
& \stackrel{(i)}{\leq} 2 \EE\Big[ \sup_{\|\tarplain - \tarstar \|_2
    \leq \myrad} \Big| \frac{1}{\numobs} \sum_{i=1}^\numobs \rad_i
  \Big\{ (\tarplain(\Covariate_i) - \auxstar(\Covariate_i,
  \Surrogate_i))^2 - (\tarstar(\Covariate_i) - \auxstar(\Covariate_i,
  \Surrogate_i))^2 \Big\} \Big| \Big] \\
& \stackrel{(ii)}{\leq} 16 \cdot \EE \Big[ \sup_{\|\tarplain
    - \tarstar \|_2 \leq \myrad} \Big| \frac{1}{\numobs}
  \sum_{i=1}^\numobs \rad_i \big(\tarplain(\Covariate_i) -
  \tarstar(\Covariate_i) \big) \Big| \Big] \\
& = 16 \cdot \radcomp(\myrad; \TarClassStar)
\end{align*}
Here step (i) follows by a standard symmetrization argument, and step
(ii) uses the Ledoux-Talagrand contraction inequality and the fact
that
\begin{align*}
\Big| (\tarplain(\covariate) - \auxstar(\covariate, \surrogate))^2 -
(\tarplain'(\covariate) - \auxstar(\covariate, \surrogate))^2 \Big|
\leq 4 | \tarplain(\covariate) - \tarplain'(\covariate) |.
\end{align*}

To complete the proof, it suffices to show that $\radcomp(\myrad;
\TarClassStar) \leq \frac{\myrad \radcrit}{16}$ for $\myrad \geq
\radcrit$.  It is known that the function $\myrad \mapsto
\frac{\radcomp(\myrad; \TarClassStar)}{\myrad}$ is non-increasing
(e.g.,see Lemma 13.6 in the book~\cite{dinosaur2019}).
Consequently, for $\myrad \geq \radcrit$, we have
$\frac{\radcomp(\myrad; \TarClassStar)}{\myrad} \leq
\frac{\radcomp(\radcrit; \TarClassStar)}{\radcrit} \leq
\frac{\radcrit}{16}$, and re-arranging establishes the required bound.


\subsection{Proof of~\Cref{LemTermTwo}}
\label{AppProofLemTermTwo}

Introduce the shorthand $\noise \defn \Response - \auxstar(\Covariate,
\Surrogate)$, and define the supremum
\begin{align*}
\NewEmp_\lnumobs(\myrad) \defn \sup \limits_{\|\tarplain - \tarstar
  \|_2 \leq \myrad} \big| \frac{1}{\lnumobs} \sum_{i=1}^\lnumobs
\noise_i \big(\tarplain(\Covariate_i) - \tarstar(\Covariate_i) \big)
\big|.
\end{align*}
By construction, we have $\PempLab \Big\{ (\tarhat -
\tarstar)(\response - \auxstar) \Big\} \leq
\NewEmp_\lnumobs(\|\tarhat - \tarstar\|_2)$.  As in our previous proof
(see~\Cref{AppProofLemTermOne}), we proceed via two steps.  First, we
establish that for each fixed $\myrad$,
\begin{align}
\label{EqnTermTwoInter}
\NewEmp_\lnumobs(\myrad) \leq \frac{\sigbound \lradcrit \myrad}{2} +
\sigbound \sqrt{\frac{2\log(1/\pardelta)}{\lnumobs}} t + \frac{8
  \sigbound \log(1/\pardelta)}{\lnumobs}.
\end{align}
By applying~\Cref{AppLemPeeling} with $Q(\myrad,s) = \myrad \cdot
\big(\tfrac{\sigbound \lradcrit}{2} + \sigbound
\sqrt{\tfrac{2s}{\lnumobs}} \big) + \frac{8\sigbound^2s}{\lnumobs}$ and $U = \max\{\lradcrit, \|\tarhat - \tarstar\|_2\}$, we
obtain
\begin{align*}
\NewEmp_\lnumobs(\|\tarhat - \tarstar\|_2) \leq \NewEmp_\lnumobs(U) \leq \sigbound \lradcrit \cdot
U + 2\sigbound
\sqrt{\frac{2\log(\Logfun(\lradcrit)/\pardelta)}{\lnumobs}} \cdot U + \frac{16\sigbound
  \log(\Logfun(\lradcrit)/\pardelta)}{\lnumobs}.
\end{align*}
Therefore we have, using the facts that $\sqrt{\lnumobs} \lradcrit
\leq \sqrt{\numobs} \radcrit$ and $\Logfun$ is a decreasing function,
we have
\begin{align*}
\Term_2 &\leq \max\{\radcrit, \|\tarhat - \tarstar\|_2\} \cdot \Big
(2 \sigbound \radcrit + 4\sigbound \sqrt{\frac{2
    \log(\Logfun(\radcrit)/\pardelta)}{\numobs}} \Big) + \frac{32
  \sigbound \log(\Logfun(\radcrit)/\pardelta)}{\numobs},
\end{align*}
as desired.


\paragraph{Proof of the bound~\eqref{EqnTermTwoInter}:}
Fix $t \geq \lradcrit$ and define $\hclass \defn \{ \noise(\tarplain -
\tarstar) : \tarplain \in \TarClass, \, \| \tarplain - \tarstar \|_2
\leq t\}$.  Applying~\Cref{lems:emp-conc} yields
\begin{align*}
\NewEmp_\lnumobs(\myrad) \leq 2 \EE\big[ \NewEmp_\lnumobs(\myrad)
  \big] + \sqrt{\sigma^2(\hclass)}
\sqrt{\frac{2\log(1/\pardelta)}{\lnumobs}} + \frac{8\sigbound^2
  \log(1/\pardelta)}{\lnumobs}
\end{align*}
with probability at least $1 - \pardelta$. As before, it suffices to
show that $\EE\big[ \NewEmp_\lnumobs(\myrad)\big] \leq \frac{\sigbound
  \myrad \lradcrit}{4}$ and $\sigma^2(\hclass) \leq \sigbound^2
\myrad^2$ for $\myrad \geq \lradcrit$.

\noindent \emph{Variance bound:} We have
\begin{align*}
\VarFun{\hclass} = \sup_{\|\tarplain - \tarstar \|_2 \leq \myrad}
\Var(\noise(\tarplain - \tarstar)) = \sup_{\| \tarplain - \tarstar
  \|_2 \leq \myrad} \EE[\noise^2 \big(\tarplain(\covariate) -
  \tarstar(\covariate)\big)^2] \leq \sigbound^2 \myrad^2.
\end{align*}

\noindent \emph{Mean bound:} We now bound the mean.  Letting
$\{\rade_i\}_{i=1}^\lnumobs$ be an i.i.d. sequence Rademacher random
variables, a symmetrization argument yields
\begin{align*}
\EE \Big[ \sup_{\|\tarplain - \tarstar\|_2 \leq \myrad} \big|
  \frac{1}{\lnumobs} \sum_{i=1}^\lnumobs \noise_i \big(
  \tarplain(\Covariate_i) - \tarstar(\Covariate_i) \big) \big| \Big] &
\leq 2 \; \EE \, \Big[ \sup_{\|\tarplain - \tarstar\|_2 \leq \myrad}
  \big| \frac{1}{\lnumobs} \sum_{i=1}^\lnumobs \rade_i \noise_i
  \big(\tarplain(\Covariate_i) - \tarstar(\Covariate_i) \big) \big|
  \Big].
\end{align*}
Then by conditioning on $(\noise_i,\Covariate_i)_{i=1}^\lnumobs$ and
applying the Ledoux-Talagrand contraction inequality to $\phi_i(t) =
\frac{\noise_i t}{\sigbound}$, we have
\begin{align*}
\EE \Big[ \sup_{\|\tarplain - \tarstar\|_2 \leq \myrad} \Big|
  \frac{1}{\lnumobs} \sum_{i=1}^\lnumobs \rade_i \noise_i\big(
  \tarplain(\Covariate_i) - \tarstar(\Covariate_i) \big) \Big| \Big]
\leq 2 \sigbound \cdot \mathcal{R}_\lnumobs(\myrad; \TarClassStar).
\end{align*}
If $\myrad \geq \lradcrit$, we have $\mathcal{R}_\lnumobs(\myrad
;\TarClassStar) \leq \frac{\myrad \lradcrit}{16}$, and putting
together the pieces establishes the claim.


\subsection{Proof of~\Cref{LemTermThree2}}
\label{sec:proof-LemTermThree2}

This proof is mostly the same as the proof of~\Cref{LemTermTwo}; the
main distinction is that we need to take extra care in verifying that
the defined empirical process is zero-mean.  Introducing the shorthand
$\omega_i \defn \auxhat(\Covariate_i, \Surrogate_i) -
\ftil(\Covariate_i) + \tarstar(\Covariate_i) - \auxstar(\Covariate_i,
\Surrogate_i)$, define the supremum
\begin{align*}
\NewEmp_\unumobs(t) \defn \sup_{\|\tarplain - \tarstar\|_2 \leq t}
\Big| \frac{1}{\unumobs} \sum_{i=1}^\unumobs \omega_i \big(
\tarplain(\Covariate_i) - \tarstar(\Covariate_i) \big) \Big|.
\end{align*}
By construction, we have $\Term_3 \leq \frac{2\unumobs}{\numobs}
\NewEmp_\unumobs(\|\tarhat - \tarstar \|_2)$. Like previously, we
begin by establishing that for each fixed $t$,
\begin{align}
\label{EqnTermThree2Inter}
\NewEmp_\unumobs(\myrad) \leq  t \uradcrit + 4 
\sqrt{\frac{2\log(1/\pardelta)}{\unumobs}} t + \frac{32 ^2
  \log(1/\pardelta)}{\unumobs}.
\end{align}
Therefore, using $Q(t, s) = t \cdot(  \uradcrit + 4
\sqrt{\tfrac{2s}{\unumobs}}) + \tfrac{32 s}{\unumobs}$ and $U = \max\{\|\tarhat - \tarstar\|_2, \uradcrit\}$
in~\Cref{lems:peeling}, we have 
\begin{align*}
\NewEmp_\unumobs(\|\tarhat - \tarstar\|_2) \leq 2 \max\{\|\tarhat - \tarstar\|_2, \uradcrit\} \Big( \uradcrit + 4 
\sqrt{\frac{2\log(\Logfun(\uradcrit)/\pardelta)}{\unumobs}} \Big) +
\frac{64 \log(\Logfun(\uradcrit)/\pardelta)}{\unumobs},
\end{align*}
with probability at least $1 - \pardelta$.  Using calculations similar
to our earlier ones along with the fact that $\sqrt{\unumobs}
\uradcrit \leq \sqrt{\numobs} \radcrit$, we find that
\begin{align*}
\frac{2 \unumobs}{\numobs} \cdot \NewEmp_\unumobs(\|\tarhat -
\tarstar\|_2) \leq 4 \|\tarhat - \tarstar \|_2 \Big( \radcrit + 4
\sqrt{\frac{2 \log(\Logfun(\radcrit)/\pardelta)}{\numobs}} \Big) +
\frac{128 \log(\Logfun(\radcrit)/\pardelta)}{\numobs},
\end{align*}
as desired.

\paragraph{Proof of~\Cref{EqnTermThree2Inter}:}

For a fixed $t \geq \uradcrit$, we define
\begin{align*}
\hclass \defn \{ \omega(\tarplain - \tarstar) : \tarplain \in
\TarClass, \, \|\tarplain - \tarstar\|_2 \leq t \}.
\end{align*}
Our approach is to apply~\Cref{lems:emp-conc} so as to control
$\NewEmp_\unumobs(t)$. We begin by verifying that
$\NewEmp_\unumobs(t)$ is the supremum of zero-mean stochastic
process. Since $\LabelSet$ and $\UnlabelSet$ are independent, we can
condition on $\LabelSet$ and accordingly treat $\auxhat$ as a fixed
quantity. Thus we have
\begin{align*}
\EE\big[\omega \big( \tarplain(\Covariate) - \tarstar(\Covariate))
  \big)\big] &= \EE \Big[ \Big( \auxhat(\Covariate, \Surrogate) -
  \ftil(\Covariate) \Big) \big( \tarplain(\Covariate) -
  \tarstar(\Covariate)) \Big] \\ & \qquad \qquad + \EE\Big[ \Big(
  \tarstar(\Covariate) - \auxstar(\Covariate, \Surrogate) \Big)
  \big(\tarplain(\Covariate) - \tarstar(\Covariate)) \Big].
\end{align*}
Observing that \mbox{$\EE[ \auxhat(\Covariate, \Surrogate) -
    \ftil(\Covariate) \mid \Covariate ] = \ftil(\Covariate) -
  \ftil(\Covariate) = 0$} and
\begin{align*}
  \EE[ \tarstar(\Covariate) - \auxstar(\Covariate, \Surrogate) \mid
    \Covariate] = \tarstar(\Covariate) - \tarstar(\Covariate) = 0,
\end{align*}
we complete the proof by applying the law of total expectation.

Applying~\Cref{lems:emp-conc} yields
\begin{align*}
\NewEmp_\unumobs(t) \leq 2 \EE[\NewEmp_\unumobs(t)] +
\sqrt{\sigma^2(\hclass)} \sqrt{\frac{2\log(1/\pardelta)}{\unumobs}} +
\frac{32 \log(1/\pardelta)}{\unumobs}.
\end{align*}
As before, we proceed to control the mean and variance terms. We have
\begin{align*}
\EE[\NewEmp_\unumobs(t)] &= \EE\Big[ \sup_{\| \tarplain - \tarstar
    \|_2 \leq t} \Big| \frac{1}{\unumobs} \sum_{i=\lnumobs+1}^\numobs
  \omega_i \big( \tarplain(\Covariate_i) - \tarstar(\Covariate_i)
  \big) \Big| \Big] \\
& \stackrel{(i)}{\leq} 2\, \EE\Big[ \sup_{\| \tarplain - \tarstar
    \|_2 \leq t} \Big| \frac{1}{\unumobs} \sum_{i=\lnumobs+1}^\numobs
  \omega_i \rad_i \big( \tarplain(\Covariate_i) -
  \tarstar(\Covariate_i) \big) \Big| \Big]
\\
& \stackrel{(ii)}{\leq} 8 \, \EE\Big[ \sup_{\| \tarplain -
    \tarstar \|_2 \leq t} \Big| \frac{1}{\unumobs}
  \sum_{i=\lnumobs+1}^\numobs \rad_i \big( \tarplain(\Covariate_i) -
  \tarstar(\Covariate_i) \big) \Big| \Big] = 8 \cdot
\mathcal{R}_\unumobs(t; \TarClassStar).
\end{align*}
Step (i) follows from a standard symmetrization argument, and step
(ii) follows from conditioning on $(\omega_i, \Covariate_i,
\Surrogate_i)$ and applying the Ledoux-Talagrand contraction
inequality with the functions $\phi_i(t) = \frac{\omega_i
  t}{4}$. By the definition of the critical radius $\uradcrit$,
we have the bound \mbox{$8  \mathcal{R}_\unumobs(\uradcrit;
  \TarClassStar) \leq \frac{ t \uradcrit}{2}$.}

It remains to control the variance.  We have
\begin{align*}
\sigma^2(\hclass) &= \sup_{\| \tarplain - \tarstar \|_2 \leq t}
\text{Var}(\omega (\tarplain - \tarstar)) \leq \sup_{\|\tarplain -
  \tarstar \|_2 \leq t} \EE\Big[ \omega^2 \big( \tarplain(\Covariate)
  - \tarstar(\Covariate))^2 \Big] \leq 16 t^2,
\end{align*}
as desired.

\subsection{Proof of~\Cref{LemTermThree}}
\label{AppProofLemTermThree}

\newcommand{\supRVm}[1]{Z_\unumobs(#1)}

The proof follows the same steps as the previous section. Define the
random variable
\begin{align*}
\supRVm{\myrad} \defn \sup_{\|\tarplain - \tarstar \|_2 \leq \myrad}
\Big| \|\tarplain - \tarstar \|_{\unumobs}^2 - \|\tarplain - \tarstar
\|_2^2 \Big|.
\end{align*}
For each $\tarplain \in \TarClass$, define the function
$\CompFun_\tarplain(\covariate) \defn \big(\tarplain(\covariate)-
\tarstar(\covariate)\big)^2 - \| \tarplain - \tarstar\|_2^2$, along
with the associated function class $\hclass(\myrad) \defn \big \{
\CompFun_\tarplain \, \mid \, \|\tarplain - \tarstar \|_2 \leq \myrad
\big\} $.  As before, we adopt the shorthand $\hclass$ when $\myrad$
is clear from context.

Applying~\Cref{lems:emp-conc} to $\hclass$ yields
\begin{align*}
\supRVm{\myrad} \leq 2 \EE[\supRVm{\myrad}] + \sqrt{\VarFun{\hclass}}
\sqrt{\frac{2\log(1/\pardelta)}{\unumobs}} + \frac{32
  \log(1/\pardelta)}{\unumobs} \qquad \mbox{with probability at least
  $1 - \delta$.}
\end{align*}
We claim that $\EE[\supRVm{\myrad}] \leq \frac{\myrad \:
  \uradcrit}{2}$ and $\VarFun{\hclass} \leq 4 \: \myrad^2$.  We return
to prove these claims momentarily.

Thus, using $Q(\myrad, s) = \myrad \cdot ( \radcrit + 2 \sqrt{\frac{2
    s}{\unumobs}}) + \frac{32 s}{\unumobs}$ and $U = \max\{\uradcrit,
\|\tarhat - \tarstar\|_2\}$ in~\Cref{lems:peeling}, we find that
\begin{align*}
\supRVm{\|\tarhat - \tarstar \|_2} \leq 2 \max\{\uradcrit, \|\tarhat -
\tarstar \|_2 \} \Big( \uradcrit + 2 \sqrt{\frac{2
    \log(\Logfun(\uradcrit)/\pardelta)}{\unumobs}} \Big) + \frac{64
  \log(\Logfun(\uradcrit)/\pardelta)}{\unumobs}.
\end{align*}
with probability at least $1 - \pardelta$. By construction, we have
\begin{align*}
  \Big| \|\tarhat - \tarstar\|_{\unumobs}^2 - \|\tarhat -
\tarstar\|_2^2 \Big| \leq \supRVm{\|\tarhat -
  \tarstar\|_2}.
\end{align*}
Furthermore, we have
\begin{align*}
\Big| \|\tarhat - \tarstar \|_2 - \| \tarhat - \tarstar \|_{\unumobs}
\Big| = \tfrac{\big| \|\tarhat - \tarstar\|_{\unumobs}^2 - \|\tarhat -
  \tarstar\|_2^2 \big|}{\|\tarhat - \tarstar \|_{\unumobs} + \|\tarhat
  - \tarstar\|_2} \leq \tfrac{\big| \|\tarhat -
  \tarstar\|_{\unumobs}^2 - \|\tarhat - \tarstar\|_2^2 \big|}{
  \|\tarhat - \tarstar\|_2}.
\end{align*}
Using the fact that $\|\tarhat - \tarstar\|_2 \geq \radcrit$, we find
that
\begin{align*}
\Big| \|\tarhat - \tarstar \|_2 - \| \tarhat - \tarstar \|_{\unumobs}
\Big| \leq 2 \max\Big\{\frac{\uradcrit}{\radcrit}, 1 \Big\} \Big(
\uradcrit + 2 \sqrt{\frac{2
    \log(\Logfun(\uradcrit)/\pardelta)}{\unumobs}} \Big) + \frac{64
  \log(\Logfun(\uradcrit)/\pardelta)}{\unumobs \radcrit}.
\end{align*}
Thus, we have
\begin{align*}
\frac{2\unumobs}{\numobs} \cdot \|\tarhat - \tarstar\|_\unumobs \leq
2\|\tarhat - \tarstar\|_2 + 4 \radcrit + 8
\sqrt{\frac{2\log(\Logfun(\radcrit)/\pardelta)}{\numobs}} + \frac{128
  \log(\Logfun(\radcrit)/\pardelta)}{\numobs \radcrit},
\end{align*}
where we have used the fact that $\sqrt{\unumobs} \uradcrit \leq
\sqrt{\numobs} \radcrit$.

\smallskip

\noindent It remains to prove the variance and mean bounds claimed
previously.

\noindent \emph{Variance bound:} We have
\begin{align*}
\VarFun{\hclass} = \sup_{h \in \hclass} \VarFun{h} \leq
\sup_{\|\tarplain - \tarstar \|_2 \leq \myrad} \PP( \tarplain -
\tarstar)^4 \leq 4\sigbound^2 \sup_{\|\tarplain- \tarstar \|_2 \leq
  \myrad} \PP(\tarplain - \tarstar)^2 = 4 \myrad^2,
\end{align*}
where the second inequality follows from the fact that $\|\tarplain -
\tarstar \|_\infty \leq 2$. \\

\noindent \emph{Mean bound:} We have
\begin{align*}
\EE[\supRVm{\myrad}] &= \EE\Big[ \sup_{\|\tarplain - \tarstar \|_2
    \leq \myrad} \Big| \|\tarplain - \tarstar \|_{\unumobs}^2 - \|
  \tarplain - \tarstar \|_2^2 \Big|\Big] \\
& \leq 2\, \EE\Big[ \sup_{\|\tarplain - \tarstar \|_2 \leq \myrad}
  \Big| \frac{1}{\unumobs} \sum_{i=1}^\unumobs\rad_i \Big( \tarplain(\Covariate_i) -
  \tarstar(\Covariate_i)\Big)^2 \Big| \Big] \\
& \leq 8  \, \EE \Big[ \sup_{\|\tarplain - \tarstar \|_2 \leq
    \myrad} \Big| \frac{1}{\unumobs} \sum_{i=1}^\unumobs \rad_i \Big( \tarplain(\Covariate_i) -
  \tarstar(\Covariate_i)\Big) \Big| \Big] \\
& = 8 \, \mathcal{R}_\unumobs(\myrad; \TarClassStar) \leq \frac{ \:
  \myrad \: \uradcrit}{2}.
\end{align*}
using the definition of the critical radius in the last step.



\section{Proof of~\Cref{thm:main}(a)}
\label{sec:proof-main}

In this appendix, we prove our guarantee for general loss functions,
as stated in part (a) of~~\Cref{thm:main}.  (See~\Cref{AppGLM} for the
proof of part (b).)

\subsection{Main argument}
For convenience, we introduce the shorthand $\LossDiff_g \defn
\Loss_{\tarhat, g} - \Loss_{\tarstar, g}$.  By the
$\scparam$-convexity property of $\Loss$, we have $\frac{\scparam}{2}
\|\tarhat - \tarstar \|_2^2 \leq \PP \LossDiff_{\auxstar}$. Thus, we
can write
\begin{align*}
  \frac{\scparam}{2} \|\tarhat - \tarstar\|_2^2 & \leq (\PP -
  \PP_\numobs)(\LossDiff_{\auxstar}) +
  \PP_\numobs(\LossDiff_{\auxstar}) \leq \underbrace{(\PP -
    \PP_\numobs)(\LossDiff_{\auxstar})}_{\revdefn \Term_1} +
  \underbrace{\PP_\numobs \big(\LossDiff_{\auxstar} -
    \LossDiff_{\auxhat} \big)}_{\revdefn \Term_2}
\end{align*}
where the last step uses the fact that
$\PP_\numobs(\LossDiff_{\auxhat}) \leq 0$ by the definition of
$\tarhat$ and the fact that $\ResponseTil_i = \auxhat(\Covariate_i)$
for both $\LabelSet$ and $\UnlabelSet$. \\

\noindent Next we introduce an auxiliary result that bounds $\Term_1$:
\begin{lemma}
\label{lems:unif-cont-2}
Conditional on $\|\tarhat - \tarstar \|_2 \geq \radcrit$, we have
\begin{align*}
\Term_1 \leq L\radcrit \| \tarhat - \tarstar \|_2 + L \|\tarhat -
\tarstar\|_2 \sqrt{\frac{8\log(\Logfun(\radcrit)/\pardelta)}{\numobs}}
+ \frac{32L\sigbound \log(\Logfun(\radcrit)/\pardelta)}{\numobs}
\end{align*}
with probability at least $1 - \pardelta$.
\end{lemma}
\noindent See~\Cref{sec:proof-lems-unif-cont-1} for the proof. \\

\noindent In addition, we claim that
\begin{align}
\label{EqnDunkin}
\Term_2 \leq 2 L \: \PP_\numobs|\auxhat - \auxstar|.
\end{align}
We return to prove this claim at the end of this section. \\

\medskip

With~\Cref{lems:unif-cont-2} and the bound~\eqref{EqnDunkin} in hand,
let us complete the proof of~\Cref{thm:main}(b).  Either we have
$\|\tarhat - \tarstar \|_2 \geq \radcrit$, or the bound
in~\Cref{lems:unif-cont-2} applies. Re-arranging this bound yields
\begin{align*}
\| \tarhat - \tarstar \|_2^2 - \Ltil \Big( \radcrit +
\sqrt{\frac{8\log(\Logfun(\radcrit)/\pardelta)}{\numobs}} \Big)
\|\tarhat - \tarstar \|_2 & \leq 2\Ltil \PP_{\numobs} |\auxhat -
\auxstar| + \frac{32\Ltil B
  \log(\Logfun(\radcrit)/\pardelta)}{\numobs},
\end{align*}
where we make use of the shorthand $\Ltil \defn
\frac{2L}{\gamma}$. Using the shorthand $b = \Ltil \Big( \radcrit +
\sqrt{\frac{8\log(\Logfun(\radcrit)/\pardelta)}{\numobs}} \Big)$, we
obtain
\begin{align*}
\Big(\| \tarhat - \tarstar \|_2 - \frac{b}{2}\Big)^2 \leq
\frac{b^2}{4} + 2\Ltil \PP_{\numobs} |\auxhat - \auxstar| +
\frac{32\Ltil B \log(\Logfun(\radcrit)/\pardelta)}{\numobs}.
\end{align*}
Since $\sqrt{u + v} \leq \sqrt{u} + \sqrt{v}$ for non-negative $u, v
\geq 0$, we conclude that
\begin{align*}
\|\tarhat - \tarstar \|_2 \leq \frac{2L}{\gamma} \, \radcrit +
\sqrt{\frac{8L}{\gamma} \cdot \PP_{\numobs} |\auxhat - \auxstar|} + 12
\sqrt{\frac{LB}{\gamma}} \cdot
\sqrt{\frac{\log(\Logfun(\radcrit)/\pardelta)}{\numobs}}.
\end{align*}
Thus, we have
\begin{align*}
\| \tarhat - \tarstar \|_2 \leq \max\Big\{ \radcrit,
\frac{2L}{\gamma} \, \radcrit + \sqrt{\frac{8L}{\gamma} \cdot
  \PP_{\numobs} |\auxhat - \auxstar|} + 12 \sqrt{\frac{LB}{\gamma}}
\cdot
\sqrt{\frac{\log(\Logfun(\radcrit)/\pardelta)}{\numobs}}\Big\},
\end{align*}
with probability at least $1 - \pardelta$. The stated form of the
claim follows from the fact that $\max \{u, v\} \leq u + v$ for
non-negative scalars $u$ and $v$.


\paragraph{Proof of~\Cref{EqnDunkin}:} Rearranging yields
\begin{align*}
\Term_2 = \PP_{\numobs} \loss_{\tarhat, \auxstar} - \PP_{\numobs}
\loss_{\tarhat, \auxhat} + \PP_{\numobs} \loss_{\tarstar, \auxhat} -
\PP_{\numobs} \loss_{\tarstar, \auxstar}.
\end{align*}
Using the Lipschitz property of $\Loss$, we have
\begin{align*}
\Big| \PP_{\numobs} \loss_{\tarhat, \auxstar} - \PP_{\numobs}
\loss_{\tarhat, \auxhat} \Big| & = \Big|
\frac{1}{\numobs}\sum_{i=1}^\numobs \Big\{
\loss(\widehat{\tarplain}(\Covariate_i), \auxstar(\Covariate_i,
\Surrogate_i)) - \loss(\widehat{\tarplain}(\Covariate_i),
\auxhat(\Covariate_i, \Surrogate_i)) \Big\} \Big| \\
& \leq \frac{1}{\numobs} \sum_{i=1}^\numobs \Big|
\loss(\widehat{\tarplain}(\Covariate_i), \auxstar(\Covariate_i,
\Surrogate_i)) - \loss(\widehat{\tarplain}(\Covariate_i),
\auxhat(\Covariate_i, \Surrogate_i)) \Big| \\
& \leq \frac{L}{\numobs} \sum_{i=1}^\numobs \Big|
\auxhat(\Covariate_i, \Surrogate_i) - \auxstar(\Covariate_i,
\Surrogate_i) \Big| = L\, \PP_{\numobs} |\auxhat - \auxstar|.
\end{align*}
By the same argument, we have $\Big| \PP_{\numobs} \loss_{\tarstar,
  \auxhat} - \PP_{\numobs} \loss_{\tarstar, \auxstar} \Big| \leq L
\, \PP_n |\auxhat - \auxstar|$, as claimed.


\subsection{Proof of~\Cref{lems:unif-cont-2}}
\label{sec:proof-lems-unif-cont-1}

For a given function $g$ and a radius $\myrad > 0$, define the random
variable
\begin{align*}
  \supRVn{\myrad} \defn \sup_{ \substack{ \tarplain \in \Tarclass
      \\ \|\tarplain - \tarstar \|_2 \leq \myrad}} \big|
  \big(\PP_{\numobs} - \PP \big) (\loss_{\tarplain, g} -
  \loss_{\tarstar, g}) \big|.
\end{align*}
By definition, we have $\Term_1 \leq \supRVn{\|\tarhat -
  \tarstar\|_2}$, so that it suffices to establish the following
intermediate result:
\begin{lemma}
\label{LemGenericUnifCont}
The  random variable $  \supRVn{\|\tarhat - \tarstar \|_2}$ is upper bounded by
\begin{align*}
L \radcrit \max\{\radcrit, \|\tarhat - \tarstar \|_2\} + L
\sqrt{\frac{8\log(\Logfun(\radcrit)/\pardelta)}{\numobs}} \cdot \max
\{ \radcrit, \|\tarhat - \tarstar \|_2\} + \frac{32L
  \log(\Logfun(\radcrit)/\pardelta)}{\numobs}
\end{align*}
with probability at least $1 - \pardelta$.
\end{lemma}

\noindent We remark that this statement also holds for the empirical
process
\begin{align*}
\supRVn{\myrad} \defn \sup_{ \substack{ \tarplain \in \Tarclass
    \\ \|\tarplain - \tarstar \|_2 \leq \myrad}} \Big|
\big(\PP_{\numobs} - \PP \big) (\loss_{\tarplain, \Response} -
\loss_{\tarstar, \Response}) \Big|,
\end{align*}
as the proof does not change if we replace $g(\Covariate, \Surrogate)$
with $\Response$.


\begin{proof}
With this notation, our first step is to apply~\Cref{lems:emp-conc} to
the function class
\begin{align*}
\hclass \defn \Big\{ \loss_{\tarplain, g} - \loss_{\tarstar,
  g} - \PP(\loss_{\tarplain, g} - \loss_{\tarstar,
  g}) \, \mid \, \tarplain \in \gclass, \enskip \| f - \tarstar
\|_2 \leq \myrad \Big\}.
\end{align*}
Setting $\tau = 1$ in our application of this lemma yields
\begin{align*}
\supRVn{\myrad} \geq 2 \EE[\supRVn{\myrad}] + \sqrt{\VarFun{\hclass}}
\sqrt{\frac{2 s}{\numobs}} + \frac{16L s}{\numobs} \qquad
\mbox{with probability at most $e^{-s}$.}
\end{align*}
We claim (and prove momentarily) that
\begin{align}
\label{eqn:controlling-uniform-terms}
\EE[\supRVn{\myrad}] \leq \tfrac{1}{4} L \, \myrad \, \radcrit, \qquad
  \text{and} \qquad \VarFun{\hclass} \leq \Lip^2 \myrad^2 \qquad
  \mbox{for all $\myrad \geq \radcrit$.}
\end{align}
When this bound holds, we have
\begin{align*}
\supRVn{\myrad} \geq Q(\myrad, s) \defn \frac{L \myrad \radcrit}{2} +
L \myrad \sqrt{\frac{2 s}{\numobs}} + \frac{16 L s}{\numobs} 
\end{align*}
with probability controlled by $e^{-s}$, as
desired. \Cref{lems:unif-cont-2} follows from
applying~\Cref{lems:peeling} to $\supRVn{\myrad}$ with $U = \max\{\radcrit, \|\tarhat - \tarstar\|_2\}$.

\paragraph{Proof of the bounds~\eqref{eqn:controlling-uniform-terms}: }

We begin by controlling the supremum $\VarFun{\hclass}$ of variances.
In particular, using the $\Lip$-Lipschitz property of the loss in its
second argument, we have
\begin{align*}
\VarFun{\hclass} &= \sup_{h \in \hclass} \VarFun{h} \leq
\sup_{\|\tarplain - \tarstar\|_2 \leq \myrad} \PP(\loss_{\tarplain, g}
- \loss_{\tarstar, g})^2 \leq L^2 \sup_{\| \tarplain - \tarstar\|_2
  \leq \myrad} \PP(\tarplain - \tarstar)^2 = L^2 \, \myrad^2,
\end{align*}
as claimed.

Shifting our focus to the expectation $\EE[\supRVn{\myrad}]$, we have
\begin{align*}
\EE[\supRVn{\myrad}] &= \EE \Big[ \sup_{\|\tarplain - \tarstar
    \|_2\leq \myrad} \big| \big(\PP_{\numobs} - \PP \big)
  (\loss_{\tarplain, g} - \loss_{\tarstar, g}) \big| \Big] \\
& \stackrel{(i)}{\leq} 2 \, \EE \Big[ \sup_{\|\tarplain - \tarstar
    \|_2 \leq \myrad} \big| \empavg \rade_i
  \big\{\loss(\tarplain(\Covariate_i),g(\Covariate_i, \Surrogate_i)) -
  \loss(\tarstar(\Covariate_i), g(\Covariate_i, \Surrogate_i))\big\}
  \big| \Big] \\
& \stackrel{(ii)}{\leq} 4 L \cdot \EE \Big[ \sup_{\|\tarplain -
    \tarstar \|_2 \leq \myrad} \big| \empavg \rade_i
  \big\{\tarplain(\Covariate_i) - \tarstar(\Covariate_i) \big\} \big|
  \Big] \\
& = 4 L \cdot \radcomp(\myrad; \TarClassStar)  \stackrel{(iii)}{\leq} \tfrac{1}{4} \, \Lip \, \myrad \: \radcrit.
\end{align*}
Here step (i) uses a standard symmetrization argument, and step (ii)
uses the Ledoux-Talagrand contraction inequality and the fact that
$\loss$ is $L$-Lipschitz.  (For further details, see the proof of
Theorem 4.10 and Equation (5.61), respectively, in the
book~\cite{dinosaur2019}). As for step (iii), since the function
$\myrad \mapsto \frac{\radcomp(\myrad; \TarClassStar)}{\myrad}$ is
non-increasing (see Lemma 13.6 in the book~\cite{dinosaur2019}),
for $\myrad \geq \radcrit$, we have \mbox{$\frac{\radcomp(\myrad;
    \TarClassStar)}{\myrad} \leq \frac{\radcomp(\radcrit;
    \TarClassStar)}{\radcrit} \leq \frac{\radcrit}{16}$,} using the
definition of $\radcrit$ in the final step.
\end{proof}


\section{Proof of~\Cref{thm:main}(b)}
\label{AppGLM}

In this section, we prove the improved bound~\eqref{eqn:faster-rate}
for GLM-type losses, as stated in part (b) of~\Cref{thm:main}.

\subsection{Main argument}

In addition to our previously introduced notation $\LossDiff_g \defn
\Loss_{\tarhat, g} - \Loss_{\tarstar, g}$, we also adopt the shorthand
$\LossDiff_\Response \defn \Loss_{\tarhat, \Response} -
\Loss_{\tarstar, \Response}$.  As in our previous analysis, we have
the upper bound $\frac{\scparam}{2} \|\tarhat - \tarstar\|_2^2 \leq
\PP \LossDiff_\auxstar$. Our first lemma decomposes the quantity $\PP
\LossDiff_\auxstar$ into a sum of three terms that we then control
individually.
\begin{lemma}
\label{LemGLMDecomp}
We have the upper bound $\PP \LossDiff_\auxstar \leq \Term_1 + \Term_2
+ \Term_3$, where
\begin{subequations}
\begin{align}
\Term_1 &\defn  \frac{\unumobs}{\numobs} \|\tarhat - \tarstar \|_2 \cdot \GLMbias, \\
\Term_2 &\defn \frac{\lnumobs}{\numobs} \big(\PP - \PempLab) \: \LossDiff_{\Response}, \qquad \text{and} \qquad \\
\Term_3 &\defn \frac{\unumobs}{\numobs} \big(\PP - \PempUnlab) \: \LossDiff_{\auxhat}. \
\end{align}
\end{subequations}
\end{lemma}

\noindent See~\Cref{sec:proof-LemGLMDecomp} for a proof of this
lemma. We now introduce a lemma to control these terms, beginning with
terms $\Term_2$ and $\Term_3$:
\begin{lemma}[Bounds on $\Term_2$ and $\Term_3$]
\label{LemGenUnifCont}
Conditional on $\|\tarhat - \tarstar \|_2 \geq \radcrit$,  we have
\begin{subequations}
\begin{align}
\Term_2 & \leq L\radcrit \| \tarhat - \tarstar \|_2 + L \|\tarhat -
\tarstar\|_2 \sqrt{\frac{8\log(\Logfun(\radcrit)/\pardelta)}{\numobs}}
+ \frac{32L \log(\Logfun(\radcrit)/\pardelta)}{\numobs}, \quad \mbox{and} \\
\Term_3 & \leq L\radcrit \| \tarhat - \tarstar \|_2 + L \|\tarhat -
\tarstar\|_2 \sqrt{\frac{8\log(\Logfun(\radcrit)/\pardelta)}{\numobs}}
+ \frac{32L \log(\Logfun(\radcrit)/\pardelta)}{\numobs},
\end{align}
\end{subequations}
each with probability at least $1 - \pardelta$.
\end{lemma}
\noindent
See~\Cref{sec:proof-LemGenUnifCont} for a proof of this result. \\

Using the shorthand $b = L \big( \radcrit + \sqrt{\frac{8
    \log(\Logfun(\radcrit)/\pardelta)}{\numobs}} \big)$, we have
\begin{align*}
\Term_2 \leq b \cdot \|\tarhat - \tarstar \|_2 + \frac{32L \log(\Logfun(\radcrit)/\pardelta)}{\numobs}, \quad \mbox{and} \quad \Term_3 \leq b \cdot \|\tarhat - \tarstar \|_2 + \frac{32L \log(\Logfun(\radcrit)/\pardelta)}{\numobs}.
\end{align*}
Therefore, 
\begin{align*}
\frac{\scparam}{2} \|\tarhat - \tarstar \|_2^2 &\leq \PP\LossDiff_{\auxstar} \leq \Term_1 + \Term_2 + \Term_3 \\
&\leq \GLMbias \cdot \|\tarhat - \tarstar \|_2 + 2b \cdot \|\tarhat -\tarstar \|_2 + \frac{64L \log(\Logfun(\radcrit)/\pardelta)}{\numobs} \\
&=  \Big(2b + \GLMbias \Big) \cdot \|\tarhat - \tarstar \|_2 + \frac{64L
  \log(\Logfun(\radcrit) /\pardelta)}{\numobs},
\end{align*}
with probability at least $1 - 2\pardelta$. We then complete the
square, thereby obtaining
\begin{align*}
\Big( \|\tarhat - \tarstar \|_2^2 - \frac{1}{\scparam} \big(b + \GLMbias \big) \Big)^2 \leq \frac{1}{\scparam^2} \big(b +
\GLMbias \big)^2 + \frac{128 L
  \log(\Logfun(\radcrit)/\pardelta)}{\scparam \numobs}.
\end{align*}
Following some straightforward algebra and using the fact that
$\sqrt{u + v} \leq \sqrt{u} + \sqrt{v}$ for non-negative scalars $u$
and $v$, we conclude that
\begin{align*}
\|\tarhat - \tarstar \|_2 &\leq \frac{2}{\scparam} \big( b + \GLMbias \big) + 8 \sqrt{\frac{2L
    \log(\Logfun(\radcrit)/\pardelta)}{\scparam \numobs}} \\
& = \frac{2 L}{\scparam} \: \radcrit + \frac{2}{\scparam} \: \GLMbias + \Big(\tfrac{2L}{\scparam} + 8\sqrt{\tfrac{L}{\scparam}}
\Big) \sqrt{\frac{2\log(\Logfun(\radcrit)/\pardelta)}{\numobs}},
\end{align*}
as claimed.


\subsection{Proof of~\Cref{LemGLMDecomp}}
\label{sec:proof-LemGLMDecomp}

The key property of GLM-type losses that is used here is the fact that
$\PP \Loss_{f, \Response} = \PP \Loss_{f, \auxstar}$, which implies
$\PP\LossDiff_{\auxstar} = \PP \LossDiff_{\Response}$. Furthermore,
recall that for this setting we have $\ResponseTil_i =
\auxhat(\Covariate_i, \Surrogate_i)$ for $i \in \UnlabelSet$ and
$\ResponseTil_i = \Response_i$ for $i \in \LabelSet$. Thus we have
\begin{align*}
\PP \LossDiff_{\auxstar} = \frac{\lnumobs}{\numobs} \: \PP
\LossDiff_\Response + \frac{\unumobs}{\numobs} \: \PP
\LossDiff_{\auxstar}, \quad \text{and} \quad \PP_\numobs
\LossDiff_{\ResponseTil} = \frac{\lnumobs}{\numobs} \PP_\lnumobs
\LossDiff_\Response + \frac{\unumobs}{\numobs} \PempUnlab
\LossDiff_{\auxhat},
\end{align*}
where $\PempLab$ and $\PempUnlab$ denote the empirical measures over
$\LabelSet$ and $\UnlabelSet$, respectively. Therefore, we can write
\begin{align*}
  \frac{\lnumobs}{\numobs} \PP \LossDiff_\Response +
  \frac{\unumobs}{\numobs} \PP \LossDiff_{\auxstar} & =
  \frac{\unumobs}{\numobs} \big( \PP \LossDiff_{\auxstar} - \PP
  \LossDiff_{\auxhat} \big) + \frac{\lnumobs}{\numobs} (\PP -
  \PP_\lnumobs) \LossDiff_\Response + \frac{\unumobs}{\numobs} (\PP -
  \PP_\unumobs) \LossDiff_{\auxhat} + \PP_\numobs
  \LossDiff_{\ResponseTil} \\ &\leq \frac{\unumobs}{\numobs} \big( \PP
  \LossDiff_{\auxstar} - \PP \LossDiff_{\auxhat} \big) +
  \frac{\lnumobs}{\numobs} (\PP - \PP_\lnumobs) \LossDiff_\Response +
  \frac{\unumobs}{\numobs} (\PP - \PP_\unumobs) \LossDiff_{\auxhat}
  \\ &= \frac{\unumobs}{\numobs} \big( \PP \LossDiff_{\auxstar} - \PP
  \LossDiff_{\auxhat} \big) + \Term_2 + \Term_3.
\end{align*}

It remains to show $\PP \LossDiff_{\auxstar} - \PP \LossDiff_{\auxhat}
\leq \| \tarhat - \tarstar \|_2 \cdot \| \ftil - \tarstar
\|_2$. Recall that the GLM-type loss takes the form $\loss(\yhat, y) =
-\phi(\yhat) \response + \Phi(\yhat)$, so that, for any function $g$, we have
\begin{align*}
  \LossDiff_g & = \loss(\tarhat(\Covariate), g(\Covariate,
  \Surrogate)) - \loss(\tarstar(\Covariate), g(\Covariate,
  \Surrogate)) \\ & = -\big \{ \phi\big(\tarhat(\Covariate)\big) -
  \phi\big(\tarstar(\Covariate) \big) \big \} g(\Covariate,
  \Surrogate) + \Phi(\tarhat(\Covariate)) - \Phi(\tarstar(\Covariate))
\end{align*}
We apply this fact with $g = \auxstar$ and $g = \auxhat$ and then take
the difference, thereby obtaining
\begin{align*}
  \PP \big(\LossDiff_{\auxstar} - \LossDiff_{\auxhat} \big) & =
  \EE\Big[ \big \{ \phi\big(\tarhat(\Covariate)\big) -
    \phi\big(\tarstar(\Covariate) \big) \big \} \big \{
    \auxhat(\Covariate, \Surrogate) - \auxstar(\Covariate, \Surrogate)
    \big \} \Big] \\
& \stackrel{(i)}{=} \EE\Big[\big \{ \phi\big(\tarhat(\Covariate)\big)
    - \phi\big(\tarstar(\Covariate) \big) \big \} \big \{
    \EE[\auxhat\mid \Covariate] - \EE[\auxstar \mid \Covariate] \big
    \} \Big] \\
& \stackrel{(ii)}{\leq} \| \phi(\tarhat) - \phi(\tarstar) \|_2 \;
  \GLMbias \\ & \stackrel{(iii)}{\leq} \| \tarhat - \tarstar \|_2 \;
  \GLMbias,
\end{align*} 
where step (i) follows from the law of total expectation, step (ii)
follows from the Cauchy--Schwarz inequality, and step (iii) follows
from the fact that $\phi$ is $1$-Lipschitz.

\subsection{Proof of~\Cref{LemGenUnifCont}}
\label{sec:proof-LemGenUnifCont}

By~\Cref{LemGenericUnifCont}, we have
\begin{align*}
\Big| \big(\PP - \PP_\lnumobs) \LossDiff_{\Response} \Big| & \leq L
\lradcrit \: \max\{\lradcrit, \|\tarhat - \tarstar\|_2 \} + L
\sqrt{\tfrac{8\log(\Logfun(\lradcrit)/\pardelta)}{\lnumobs}} \: \max \{
\lradcrit, \|\tarhat - \tarstar\|_2\} + \tfrac{32L
  \log(\Logfun(\lradcrit)/\pardelta)}{\lnumobs},
\end{align*}
with probability at least $1 - \pardelta$. Furthermore, by
conditioning on $\auxhat$ (which is trained on $\LabelSet$) and using
the fact that $\LabelSet$ and $\UnlabelSet$ are independent, we have
by the same lemma
\begin{align*}
\Big| \big(\PP - \PempUnlab \big) \LossDiff_{\auxhat} \Big| &\leq L
\uradcrit \: \max\{\uradcrit, \|\tarhat - \tarstar\|_2 \} + L
\sqrt{\tfrac{8\log(\Logfun(\uradcrit)/\pardelta)}{\lnumobs}} \: \max
\{ \uradcrit, \|\tarhat - \tarstar\|_2\} + \tfrac{32L
  \log(\Logfun(\uradcrit)/\pardelta)}{\lnumobs}.
\end{align*}
By using the facts that $\sqrt{\lnumobs} \lradcrit \leq \sqrt{\numobs}
\radcrit$ and $\sqrt{\unumobs} \uradcrit \leq \sqrt{\numobs}
\radcrit$, we conclude that
\begin{align*}
  \max \{ \Big| \big(\PP - \PP_\lnumobs) \LossDiff_{\Response}
  \Big|, \Big| \big(\PP - \PempUnlab \big) \LossDiff_{\auxhat}
  \Big| \} & \leq L \max\{\radcrit, \|\tarhat - \tarstar\|_2 \} \Big
  \{ \radcrit +
  \sqrt{\tfrac{8\log(\Logfun(\radcrit)/\pardelta)}{\numobs}} \Big \} +
  \tfrac{32 L \log(\Logfun(\radcrit)/\pardelta)}{\numobs},
\end{align*}
The claim then follows from using the condition that $\|\tarhat -
\tarstar \|_2 \geq \radcrit$.


\section{Auxiliary lemmas}

In this section, we state some auxiliary results used throughout the
various proofs.~\Cref{AppEmpConc} is dedicated to the main
concentration result of empirical processes used
throughout. Then,~\Cref{AppLemPeeling} proves a peeling result over
empirical processes, a key ingredient in many of the arguments.

\subsection{Concentration for empirical processes}
\label{AppEmpConc}

Define the quantities $\| \PP_\tnumobs \|_\Gclass \defn \sup_{g \in
  \Gclass} \Big| \empavg g(X_i) \Big|$ and $\VarFun{\Gclass} \defn
\sup_{g \in \Gclass} \Big\{ \empavg \Var\big(g(X_i) \big) \Big\}$.
The following lemma gives an upper tail bound on $\| \PP_\numobs
\|_\Gclass$ in terms of its expectation and a deviation term:
\begin{lemma}[\cite{klein2005empproc}]
\label{lems:emp-conc}
Consider a countable, $\sigbound$-uniformly bounded function class
$\Gclass$ such that $\EE[g(X)] = 0$ for all $g \in \Gclass$.  Then for
any $\tau > 0$, we have
\begin{align*}
\| \PP_\tnumobs \|_\gclass \leq (1 + \tau) \EE \| \PP_\tnumobs
\|_\gclass + \sqrt{\VarFun{\Gclass}}
\sqrt{\tfrac{2\log(1/\pardelta)}{\tnumobs}} + \big(3 + \tfrac{1}{\tau}
\big) \tfrac{\sigbound \log(1/\pardelta)}{\tnumobs}
\end{align*}
with probability at least $1 - \pardelta$.
\end{lemma}

\subsection{A peeling lemma}
\label{AppLemPeeling}

Given a function class $\Fclass$ with norm $\|\cdot\|$ and some
empirical process $\{V_\numobs(\genfun), \genfun \in \Fclass \}$,
consider a $r$-localized supremum of the form $\supRVn{r} \defn \sup_{
  \{f \in \Fclass \mid \|\genfun \| \leq r \}} V_{\numobs}(\genfun)$.
Suppose there is some scalar $s$ and function $Q: \real^2 \rightarrow
\real$ such that
\begin{align}
\label{EqnQbound}  
\Prob \big[ \supRVn{r} \geq Q(r, t) \big] & \leq e^{-t} \qquad \mbox{for each fixed
  $r \geq s$.}
\end{align}
In this section, we show how to prove that for a bounded random
variable $U \in [s, \fbound]$, we have $\supRVn{U} \lesssim Q(U, t)$
with high probability.  As is standard in empirical process
theory~\cite{vandeGeer,dinosaur2019}, we do so via a ``peeling''
argument. \\

\noindent More formally, we establish the following:
\begin{lemma}[Peeling]
  \label{lems:peeling}
Suppose that the tail bounds~\eqref{EqnQbound} hold for all $r \geq
s$, the bivariate function $Q$ is increasing in its first argument,
and $Q(2r, t) \leq 2 Q(r, t)$ for all $r \geq s$ and $t > 0$.  Then
for any random variable $U \in [s, \fbound]$, we have
\begin{align}
\label{EqnEricPeel}
\PP\big( \supRVn{U} \geq 2 Q(U, t) \big) \leq \lceil
\log_2(\tfrac{2\fbound}{s}) \rceil \cdot e^{-t}.
\end{align}
\end{lemma}

\begin{proof}
Define the event $\Bset_m \defn \{ 2^{m-1}s \leq U \leq 2^m s \}$;
since $U \in [s, \fbound]$ we have $\PP\big(\bigcup_{m=1}^M \Bset_m
\big) = 1$ for $M = \lceil\log_2 (\tfrac{2\fbound}{s})
\rceil$. Therefore, we have
\begin{align*}
\PP\big( \supRVn{U} \geq 2Q(U, t) \big) & \stackrel{(i)}{\leq}
\sum_{m=1}^M \PP\big(\{\supRVn{U} \geq 2Q(U, t) \} \cap \Bset_m \big)
\\ & \stackrel{(ii)}{\leq} \sum_{m=1}^M \PP \big( \{\supRVn{2^m s}
\geq 2Q(2^{m-1} s, t) \} \cap \Bset_m \big) \\ &\leq \sum_{m=1}^M \PP
\big( \supRVn{2^m s} \geq 2Q(2^{m-1} s, t) \big) \\ &
\stackrel{(iii)}{\leq} \sum_{m=1}^M \PP \big( \supRVn{2^m s} \geq
Q(2^{m} s, t) \big) \stackrel{(iv)}{\leq} M e^{-t}.
\end{align*}
Step (i) follows from a union bound, step (ii) follows from the
definition of $\Bset_m$, and steps (iii) and (iv) follows from the
definition of $Q(r,t)$.

\end{proof}

\end{document}